\documentclass[11pt, a4paper]{article}
\usepackage{amssymb,amsmath,amsthm}
\usepackage{fancyhdr}
\usepackage[british]{babel}
\usepackage{color}
\usepackage[title]{appendix}
\usepackage{amsmath,amsthm}

\usepackage{enumerate,enumitem}
\usepackage{enumitem,verbatim}

\usepackage{ulem}
\usepackage{hyperref}
\usepackage[margin=2.5cm]{geometry}
\usepackage{lineno}
\usepackage[ruled]{algorithm2e}
\usepackage{todonotes}

\newtheorem{theorem}{Theorem}[section]

\newtheorem{lemma}[theorem]{Lemma}

\newtheorem{conjecture}[theorem]{Conjecture}
\newtheorem{definition}[theorem]{Definition}

\theoremstyle{definition}

\newtheorem{question}{Question}

\newtheorem{rmk}{Remark}

\newtheorem{claim}[theorem]{Claim}
\newtheorem{fact}[theorem]{Fact}

\newcommand{\btheorem}{\begin{theorem}}
\newcommand{\etheorem}{\end{theorem}}
\newcommand{\bconjecture}{\begin{conjecture}}
\newcommand{\econjecture}{\end{conjecture}}
\newcommand{\bproposition}{\begin{proposition}}
\newcommand{\eproposition}{\end{proposition}}
\newcommand{\bdefinition}{\begin{definition}}
\newcommand{\edefinition}{\end{definition}}
\newcommand{\bcorollary}{\begin{corollary}}
\newcommand{\ecorollary}{\end{corollary}}
\newcommand{\bproof}{\begin{proof}}
\newcommand{\eproof}{\end{proof}}
\newcommand{\bclaim}{\begin{claim}}
\newcommand{\eclaim}{\end{claim}}
\newcommand{\bquestion}{\begin{question}}
\newcommand{\equestion}{\end{question}}
\newcommand{\bfact}{\begin{fact}}
\newcommand{\efact}{\end{fact}}
\newcommand{\bremark}{\begin{remark}}
\newcommand{\eremark}{\end{remark}}
\newcommand{\eexample}{\end{example}}
\newcommand{\bexample}{\begin{example}}

\newcommand{\elemma}{\end{lemma}}
\newcommand{\blemma}{\begin{lemma}}

\newcommand{\be}{\beta}

\newcommand{\eps}{\varepsilon}
\newcommand{\<}{\subseteq}

\newcommand{\De}{\Delta}
\newcommand{\de}{\delta}

\title{Clique factors in random samplings of regular graphs}

\author{
Wanting Sun\thanks{Data Science Institute, Shandong University, Jinan, China. Email: {\tt wtsun@sdu.edu.cn}.},\ \ 
Shunan Wei\thanks{School of Mathematics, Shandong University, Jinan, China. Email: {\tt snwei@mail.sdu.edu.cn}.},\ \ 
Donglei Yang\thanks{School of Mathematics, Shandong University, Jinan, China. Email: {\tt dlyang@sdu.edu.cn}.}}

\date{\today}

\begin{document}
\maketitle
\linespread{1.2}
%\linenumbers

\begin{abstract}
We show that for any integer $r\ge 2$, there exists a constant $c>0$ such that for every sufficiently large integer $n$, every $((r-1)n+1)$-regular graph $G$ on $rn$ vertices has at least $c2^{rn}$ subsets $S\subseteq V(G)$ such that $G[S]$ contains a $K_r$-factor. This confirms a conjecture of Dragani\'c, Keevash and M\"uyesser for large $n$ [\textit{Cyclic subsets in regular Dirac graphs. Int. Math. Res. Not., 2025(14): 1-16, 2025}].
\end{abstract}

%%%%%%%%%%%%%%%%%%%%%%%%%%%%%%%%%%%%%%%%%%%%%%%%%%%%%%%%%%%

\section{Introduction}
A classical line of research in extremal graph theory concerns about the  minimum degree condition that guarantees the existence of certain spanning subgraphs.   Once such a result is established,  a natural further question is to measure the \textit{robustness} of the graph with respect to the property of containing the given spanning structure. This direction of research aims to strengthen classical results in extremal and probabilistic combinatorics.

%{Leave it for Wei to rephrase}~ 
Hamiltonicity is one of the
most central problems in graph theory. A classical result proved by Dirac in 1952~\cite{Dirac} states that for $n\geq 3$, every $n$-vertex graph $G$ with $\delta(G) \ge \frac{n}{2}$ (known as \textit{Dirac graph}) contains a Hamilton cycle.  
 %Investigating when the classical theorems hold in a \textit{robust} or \textit{resilient} way has emerged as a highly active research topic. 
% with Sudakov and Vu~\cite{Sudakov2019} pioneering its systematic study by first articulating its significance.
%Determining minimum degree conditions that guarantee the existence of a spanning structure is a central problem in extremal graph theory.  
%A cornerstone result in this direction is Dirac's theorem~\cite{Dirac}, which states that for $n\geq 3$, every $n$-vertex graph $G$ with $\delta(G) \ge \frac{n}{2}$ (known as \textit{Dirac graph}) contains a Hamilton cycle.  
%Angluin and Valiant~\cite{Angluin1979} provided  a robust version of Dirac's theorem, that is, the random graph $G(n,p)$ contains a Hamilton cycle with $p\ge \Omega(\frac{\log n}{n})$. 
A natural way to strengthen Dirac’s theorem is to ask how many Hamilton cycles a Dirac graph must contain. S\'{a}rk\"{o}zy, Selkow and Szemer\'{e}di~\cite{Sarkozy2003} proved that every Dirac graph contains at least $c^nn!$ Hamilton cycles for some small positive constant $c$, and conjectured that $c$ can be improved to $\frac{1}{2}-o(1)$. This conjecture was resolved by Cuckler and Kahn~\cite{Cuckler2009}. 
{We refer the reader to a survey of Sudakov~\cite{Sudakov2017} where various measures of robustness and relevant results are collected.}
%Cuckler and Kahn \cite{Cuckler2009} proved this  by showing a stronger result that every Dirac graph with minimum degree $\delta$ contains at least $(\frac{\delta}{e+o(1)})^n$ Hamilton cycles. 
%Determining sufficient minimum degree conditions for the existence of an $H$-factor is the fundamental research in extremal graph theory.  In this direction, the cornerstone theorem of Dirac~\cite{Dirac} states that every $n$-vertex graph $G$ with minimum degree $\delta(G)\ge \frac{n}{2}$ contains a Hamilton cycle. In particular, if $n$ is even then $G$ has a perfect matching (i.e. $K_2$-factor). A graph on $n$ vertices is called a \textit{Dirac graph} if its minimum degree is at least $n/2$.
%There are many further results to illustrate the Hamiltonicity problem, such as counting the number of Hamiltonian cycles in Dirac graphs~\cite{Sarkozy2003,Cuckler2009}, the robustness or resilience of Hamiltonicity and so on. 
% (such as the random graph $G(n,p)$, which is Hamiltonian with $p\ge \Omega(\frac{\log n}{n})$; see~\cite{Angluin1979, Shamir1983}), and so on. 

% where various measures of robustness and relevant results are collected, some further results on resilience can be found in~\cite{Krivelevich2011,Montgomery2019, Sudakov2019}.

We introduce the robustness problem from the following two perspectives. 
% A widely studied robustness problem concerns the Hamiltonicity of Dirac graphs or random graphs in which the edges are sampled at random and substantial progress has been made (see~\cite{Frieze2008,Lee2012}).
Let $G$ be a graph with a property $\mathcal{P}$.
On the one hand, we sample each edge uniformly at random in $G$ and study the threshold for the property $\mathcal{P}$ of the resulting random subgraph. 
% Recently, there has been increasing interest in the study of \textit{robustness} and \textit{resilience} of graph properties, aiming to strengthen classical results in extremal and probabilistic combinatorics.
% For a comprehensive overview about robustness, we refer the reader to the survey~\cite{Sudakov2017}.
% For example, given an $n$-vertex graph with minimum degree at least $\frac{n}{2}$, Dirac theorem guarantees that there are at least one Hamilton cycle, but can one determine the number of such cycles?
% Cuckler and Kahn~\cite{Cuckler2009} answer this question and they prove that Dirac graph with $n$ vertices contains at least $(1/2-o(1))^nn!$ Hamilton cycles.
{More precisely, let $G(p)$ be a random subgraph of $G$ such that every edge in $G$ is kept independently and randomly with probability $p$.} 
Krivelevich, Lee and Sudakov~\cite{Krivelevich2014} studied the threshold for hamiltonicity and showed that there exists a constant $C$ such that for any $n$-vertex Dirac graph $G$ and $p\ge \frac{C\log n}{n}$,  %if one takes edges of $G$ independently at random with probability $p$, 
\textit{with high probability}\footnote{We say that an event $E$ holds \textit{with high probability,} (w.h.p. for short) if $\mathbb{P}[E]\rightarrow 1$ as $n$ tends to infinity.} $G(p)$ is Hamiltonian.
On the other hand, one may sample vertices rather than edges at random and investigate the probability that the induced subgraph on such a random vertex set still possesses property \(\mathcal{P}\). {In this direction, Erd\H{o}s and Faudree~\cite{Erdos1999} asked whether every regular Dirac graph contains a constant proportion of Hamiltonian vertex subsets.} 
%\sout{This was motivated by a conjecture of Erd\H{o}s and Faudree~\cite{Erdos1999} as follows.}

\begin{conjecture}[\cite{Erdos1999}]\label{conj:Hamilton cycle}
Any $(n+1)$-regular graph $G$ on $2n$ vertices has at least $c2^{2n}$ subsets of $V(G)$ inducing a Hamiltonian subgraph for some absolute $c>0$.
\end{conjecture}

{This conjecture is tight in several ways: the \((n+1)\)-regularity cannot be relaxed to \(n\)-regularity (witnessed by \(K_{n,n}\)) nor weakened to a minimum degree condition (illustrated by adding a spanning star to each class of \(K_{n,n}\)). %On the one hand, the \((n+1)\)-regularity condition cannot be replaced by \(n\)-regularity, as shown by the complete bipartite graph $K_{n,n}$. On the other hand, the regularity assumption cannot be weakened to minimum degree, which can be seen from the graph $G$ obtained by adding a spanning star in each part of $K_{n,n}$.} %We remark that the $(n+1)$-regularity assumption cannot be weakened to the minimum degree $n+1$ or $n$-regularity.  In particular, the construction of complete bipartite graph $K_{n-1,n+1}$ by adding a $2$-factor in the part of size $n+1$ uniquely minimizes the number of Hamiltonian subsets. 
Conjecture~\ref{conj:Hamilton cycle} was recently solved for large $n$ by Dragani\'c, Keevash and M\"uyesser~\cite{Keevash2025}, who showed that a uniformly random vertex subset of an $(n+1)$-regular $2n$-vertex graph induces a Hamiltonian graph with probability at least $\frac{1}{2}$. The bound  $\frac{1}{2}$ is tight, as shown by the example of the  complete bipartite graph $K_{n-1,n+1}$ with a $2$-factor added to the larger side.    %Conjecture~\ref{conj:Hamilton cycle} for large $n$ by showing that and they determined the optimal value $c$ to be precisely $\frac{1}{2}$. We can phrase Conjecture~\ref{conj:Hamilton cycle} as taking a random vertex-subset by including each vertex independently with probability $\frac{1}{2}$ and to show that the subgraph induced on this set is Hamiltonian with constant probability.}
%Wei: Yes! This is their main result.
Very recently, Hunter, Liu, Milojevi\'c and Sudakov~\cite{hunter2025}  investigated a tournament analogue of Conjecture~\ref{conj:Hamilton cycle} under a minimum semi-degree condition.} % degree and give an optimal number of subsets of $V(T)$ inducing a Hamiltonian subgraph.}
%In the setting of digraph, Ghouila and Houri~\cite{} show an analogue of Dirac theorem, that is, every $n$-vertex digraph with minimum semi-degree at least $\frac{n}{2}$ contains a Hamilton cycle.

\subsection{The Hajnal-Szemer\'edi theorem}
%\sout{Given a graph $H$ that has a connected component of size at least $3$, Hell and Kirkpatrick~\cite{Hell1983} prove that the decision problem for $H$-factors is NP-complete.}
Given graphs $G$ and $H$, an \textit{$H$-tiling} in $G$ is a collection of vertex-disjoint copies of $H$. 
A \textit{perfect $H$-tiling} (or \textit{$H$-factor}) is one that covers every vertex of $G$.  A celebrated theorem of Hajnal and Szemer\'edi~\cite{HajnalSz} provides the best possible minimum degree condition that guarantees the existence of a $K_r$-factor (the case $r=3$ was previously obtained by Corr\'adi and Hajnal~\cite{Corradi}), {which answers a conjecture of Erd\H{o}s \cite{E1964}. }%{conjecture of Erdos}
\begin{theorem}[\cite{Corradi, HajnalSz}]
\label{thm:HSz}
Let $G$ be an $n$-vertex graph with $n\in r\mathbb{N}$. If $\delta(G)\ge \left(1-\frac{1}{r}\right)n$, then $G$ contains a $K_r$-factor.
\end{theorem}

The minimum degree condition is best possible by considering slightly unbalanced complete $r$-partite graphs. Later, Kierstead and Kostochka \cite{Kierstead2008} gave a short proof of the Hajnal-Szemer\'edi theorem phrased in terms of equitable colorings. {For more recent results on equitable colorings, see \cite{cheng2025,kierstead-survey}.} %\sout{and they also prove the analogue of Hajnal-Szemer\'edi theorem under an Ore-type degree condition~\cite{Kostochka2008}}.
%\sout{Here and through this paper, we always assume the necessary condition $n\in r\mathbb{N}$, otherwise we will draw special attention to it.}
% For a general graph $H$, Alon and Yuster~\cite{} gives an asymptotic result, that is, if $\delta(G)\ge (1-\frac{1}{\chi(H)})n+o(n)$, then $G$ contains an $H$-factor, where $\chi(H)$ is the chromatic number of $H$.
% After some excellent work (see e.g.~\cite{}), it is finally solved by K\"{u}hn and Osthus up to an additive constant.
Pham, Sah, Sawhney and Simkin~\cite{Pham2} proved a random sparsification version of Theorem~\ref{thm:HSz}.  
More precisely, they showed that given an $n$-vertex graph $G$ with $\delta(G)\ge \left(1-\frac{1}{r}\right)n$, 
if $p\ge Cn^{-2/r}(\log n)^{1/{r \choose 2}}$, then w.h.p.~$G(p)$ contains a $K_r$-factor.  
The case $r=3$ was established earlier by Allen, Böttcher, Corsten, Davies, Jenssen, Morris, Roberts and Skokan~\cite{Allen}. 
{Kelly, M\"uyesser and Pokrovskiy~\cite{Kelly2024}, and independently Joos, Lang and Sanhueza-Matamala~\cite{Joos-Lang2023},  extended this result to general $F$-factors in hypergraphs, determining the  {asymptotically} optimal threshold $p$ for which  $G(p)$ contains an $F$-factor with high probability. For further results on transversal robust versions of some classical theorems, we refer the reader to~\cite{Anastos-Chakraborti, Ferber2022, Han2025}. 
%There are many results about the transversal robust version of Dirac' theorem~\cite{Ferber2022,Anastos-Chakraborti} and Theorem~\ref{thm:HSz}~\cite{Ferber2022,Han2025}. 
}%{(add more references on $G[p]$)}

By analogy with Conjecture~\ref{conj:Hamilton cycle}, one may ask how many vertex subsets $S\subseteq V(G)$ have the property that $G[S]$ contains a $K_r$-factor.   Dragani\'{c}, Keevash and M\"{u}yesser~\cite{Keevash2025} proposed the following {conjecture regarding the Hajnal-Szemer\'{e}di theorem.} %{robust version of the Hajnal–Szemer\'{e}di theorem.}
\begin{conjecture}[\cite{Keevash2025}]\label{conj:K_r factor}
For any $r\ge 2$ there is some constant $c>0$ so that if $G$ is an $((r-1)n+1)$-regular graph on $rn$ vertices, then at least $c2^{rn}$ subsets of $V(G)$ induce a $K_r$-factor.
\end{conjecture}

%\sout{They also proposed that a plausible class of extremal constructions for Conjecture~\ref{conj:K_r factor} may be  slightly unbalanced complete $r$-partite graphs perturbed with a suitable factor in the largest part.}
In the following, we make some remarks on this conjecture. %also remark that:
\begin{enumerate}[label=(\arabic*)]
    \item The regularity condition $((r-1)n+1)$ cannot be lowered to $(r-1)n$.  
Indeed, consider the balanced complete $r$-partite graph $G:=K_{n,\ldots,n}$: an induced subgraph $G[S]$ contains a $K_r$-factor exactly when $S$ contains the same number of vertices from each part. Hence the number of such subsets $S$ is %$\sum_{k=1}^n \binom{n}{k}^r\le n \big(\frac{\sum_{k=1}^n\binom{n}{k}}{n}\big)^r=\frac{2^{rn}}{n^{r-1}}$;

$$
\sum_{k=1}^n \binom{n}{k}^r \le n \cdot \binom{n}{\lfloor n/2 \rfloor}^r \leq n \left( \frac{2^n}{\sqrt{\pi n/2}} \right)^r =n\frac{2^{rn}}{(\pi n/2)^{r/2}}.
 $$
 \item The regularity assumption cannot be reduced to a minimum degree of $(r-1)n+1$. 
Consider the graph $G$ obtained from  the balanced complete $r$-partite graph $K_{n,\dots,n}$ (with vertex partition $A_1\cup \ldots \cup A_r$) by adding a spanning star  in each part.  
In $G$, a subset $S\subseteq V(G)$ induces a subgraph that contains a $K_r$-factor only when the size of %the intersection of 
each $|A_i\cap S|$ %with each part 
deviates from $\frac{|S|}{r}$ by less than  $r$\footnote{Note that each $G[A_i\cap S]$ has matching number at most one. Thus $|A_i\cap S|-\frac{|S|}{r}\leq 1$ for each $i\in [r]$, which means $\min\limits_{i\in[r]}|A_i\cap S|>\frac{|S|}{r}-r$.}. {Then the number of such subsets $S$ is at most
$$
\sum_{s_1=0}^n\sum_{\substack{|s_j-s_1|<2r\\j\in[2,r] }}\prod_{i=1}^r\binom{n}{s_i}\leq \sum_{s_1=0}^n\sum_{\substack{|s_j-s_1|<2r\\j\in[2,r] }}\binom{n}{\lfloor n/2 \rfloor}^r\leq (n+1)(4r)^{r-1}\Big( \frac{2^n}{\sqrt{\pi n/2}} \Big)^r, 
$$
where $s_i=|A_i\cap S|$ for each  $i\in [r]$. }

%{compute the number of S?}%The regularity assumption cannot be weakened to the minimum degree $(r-1)n+1$ as a balanced complete multipartite graph $G=K_{n,\cdots,n}$ with $r$ spanning stars added in each part. In this case, $G[S]$ has a $K_r$-factor only if the size of each part in $G[S]$ deviates from $\frac{|S|}{r}$ by at most $r$.
\end{enumerate}

Our main contribution is to resolve Conjecture~\ref{conj:K_r factor} for all sufficiently large $n$.

\begin{theorem}\label{thm:main thm}
For any $r\ge 2$, there is a constant $c>0$ such that the following holds for sufficiently large $n$. Let $G$ be an $((r-1)n+1)$-regular graph on $rn$ vertices. Then at least $c2^{rn}$ subsets of $V(G)$ induce a $K_r$-factor.
\end{theorem}

In fact, we prove that the constant {$c$ can be taken as  $\frac{1}{(40r^2)^r}$}. It is worth noting that the case $r=2$ (perfect matchings) of Conjecture~\ref{conj:K_r factor} follows from the result of Dragani\'{c}, Keevash and M\"{u}yesser~\cite{Keevash2025} concerning Hamilton cycles, together with the fact that a required subset must have even size, which occurs with probability $\frac{1}{2}$.
Let $G[p]$ be a random induced subgraph of $G$ with each vertex kept independently with probability $p$, where $p\in (0,1)$. 
Consequently, Theorem~\ref{thm:main thm} can be reduced to studying the probability that \(G[\frac{1}{2}]\) contains a \(K_r\)-factor. %and the key ingredient in our proof of Theorem~\ref{thm:main thm} is the following.

\begin{theorem}\label{thm:the random version of main thm}
For any $r\ge 3$, let $G$ be an $((r-1)n+1)$-regular graph on $rn$ vertices with $n$ sufficiently large. 
Then $\mathbb{P}\big[G[\frac{1}{2}]~\text{admits a}~K_r\text{-factor}\big]\ge \frac{1}{(40r^2)^r}$.
\end{theorem}
{Indeed, the result of Theorem~\ref{thm:the random version of main thm} can be extended to a more general probability $p\in (0,1)$. 
%Let $G[p]$ be a random induced subgraph with each vertex kept independently with probability $p$.
Following the same  proof of Theorem~\ref{thm:the random version of main thm}, one may show that $G[p]$ contains a $K_r$-factor with probability at least $(\frac{p^2}{20r^2})^r$.
%Since the argument is nearly identical to the proof of Theorem~\ref{thm:the random version of main thm}, we skip the details.
}

\subsection{Notation}
For a graph $G$, let $V(G)$ and $E(G)$ denote the vertex set and edge set of $G$, respectively, and denote $|G|=|V(G)|$ and $e(G)=|E(G)|$. For a vertex $v\in V(G)$, let $N_G(v)$ be the neighborhood of $v$ in $G$. For two vertex subsets $U,U'\subseteq V(G)$, let $N_G(v,U)=N_G(v)\cap U$ and $N_G(U',U)=\bigcap_{v\in U'}N_G(v,U)$. We let \(d_{G}(v)=|N_{G}(v)|\) denote the degree of \(v\) and write \(d_{G}(v,U)=|N_{G}(v,U)|\). % for the degree of \(v\) into a subset \(U\subset V(G)\). 
As we do elsewhere, we omit \(G\) from the subscript whenever there is no risk of confusion. The minimum and maximum degree of \(G\) is denoted by \(\delta(G)\) and \(\Delta(G)\) respectively.

For a subset \(U\subseteq V(G)\), let \(G[U]\) denote the graph induced by \(U\). Let \(G-U\) be the graph induced by \(V(G)\setminus U\). % {and \(G+U\) be the graph induced by \(V(G)\cup U\)}. 
Given a subset {\(U'\subset V(G)\setminus U\)}, let \(G[U,U']\) denote the bipartite graph with bipartition \(U\cup U'\) and edges of the form \(uu'\in E(G)\) with \(u\in U\) and \(u'\in U'\). Denote $e(U,U')=|E(G[U,U'])|$. %Given a subset \(U\subset V(G)\), let \(G-U\) be the graph induced by \(V(G)\setminus U\).%, and use similar natural and common other notation. 

%\sout{Given a set $U$ and an integer $k$, we write $\binom{U}{k}$ the collection of all $k$-element subsets of $U$.}
For any integers $a\leq b$, define $[a,b]:=\{i\in \mathbb{N}:a\leq i\leq b\}$ and $[b]:=[1,b]$. Given real numbers $a,b,c$, we write $a=b\pm c$ if it holds that $b-c\leq a\leq b+c$. When we write  $\alpha\ll \beta\ll \gamma$, we always mean that $\alpha, \beta, \gamma$  are constants in $(0,1)$; and $\alpha\ll \beta$ means that there exists $\alpha_0=\alpha_0(\beta)$ such that the subsequent arguments hold for all $0<\alpha\leq \alpha_0$. Hierarchies of other lengths are defined analogously. 

%\sout{Let $G$ be a graph on $rn$ vertices and  $\mathcal{P}:=\{V_1,\ldots,V_t\}$ be a vertex partition of $G$. We say $V_i$ is a \textit{large} part of $\mathcal{P}$ is $|V_i|>\frac{|G|}{r}$, and a \textit{small} part otherwise.}%  matching, vertex cover, 

\section{Overview}
Let $G$ be an $((r-1)n+1)$-regular graph on $rn$ vertices. Given $\gamma>0$ and a subset $S\subseteq V(G)$, we say $S$ is a $\gamma$-\textit{independent set} in $G$ if $e(G[S])\leq \gamma n^2$.  We divide the proof of Theorem~\ref{thm:the random version of main thm} into two parts according to whether $G$ contains a $\gamma$-independent set of size $n$ (extremal case) or not (non-extremal case). 

\begin{theorem}[Non-extremal case]\label{thm:Non-extremal case}
Suppose that $r\geq 3, n\in \mathbb{N}$ and $\frac{1}{n}\ll \gamma \ll \frac{1}{r}$. Let $G$ be an $((r-1)n+1)$-regular graph on $rn$ vertices. If $G$ contains no $\gamma$-independent set of size $n$, then $G[\frac{1}{2}]$ admits a $K_r$-factor with probability at least $\frac{1}{r}-o(1)$.
\end{theorem}

Note that an essentially necessary condition for the existence of a $K_r$-factor in a graph is that the order of the graph is divisible by $r$. % is the divisibility of $|G[\frac{1}{2}]|$ by $r$, which means 
It follows that $\mathbb{P}\big[G[\frac{1}{2}]~\text{admits a}~K_r\text{-factor}\big]\le \frac{1}{r}$. A result of Gan, Han and Hu~\cite{gan2024} (see Theorem~\ref{thm:Non-extremal case r>3}) shows that under a weaker degree condition $\delta(G)\geq (r-1)n - o(n)$, if $G$ contains no $\gamma$-independent set, then $G$ admits a $K_r$-factor. The non-extremal case essentially follows from this theorem because, with high probability, the required minimum degree and the non-existence of sparse sets would be inherited in $G[\frac{1}{2}]$. The major task is to deal with the extremal case as follows.

\begin{theorem}[Extremal case]\label{lem:extremal-case}
Suppose that $r\geq 3$, $n\in \mathbb{N}$   and $\frac{1}{n}\ll \gamma \ll \frac{1}{r}$. Let $G$ be an $((r-1)n+1)$-regular graph on $rn$ vertices which contains a $\gamma$-independent set of size $n$. Then  $G[\frac{1}{2}]$ admits a $K_r$-factor with probability at least $\frac{1}{(40r^2)^r}$.
\end{theorem}
In the extremal case, the existence of a $\gamma$-independent $n$-set in $G$ allows us to construct a partition $\mathcal{P}^0 = \{A_1^0, \dotsc, A_s^0, B^0\}$ of $V(G)$ for some $s\in [r]$, which is simply obtained by iteratively choosing disjoint $\gamma$-independent $n$-sets $A_i^0$ so that $B^0$ contains no $\gamma$-independent set {of size $n$}. 
By suitably reassigning some vertices, $\mathcal{P}^0$ can be refined into a \textit{good partition} of $G$ (see Definition~\ref{def:good partition}).  We first define the following three types of vertices with respect to a partition of $V(G)$.

\begin{definition}\label{def:bad vertex for Ai}
Let $r,t,n\in \mathbb{N}$ and $\alpha>0$. Suppose that $G$ is a graph on $rn$ vertices and   $\mathcal{P}=\{V_1,\ldots,V_t\}$ is a partition of $V(G)$. For each $i\in [t]$, we say that 
\begin{itemize}
    \item a vertex $v\in V_i$ is $(\alpha, V_i)$-good if $d(v,V_i)\le \alpha n$;
    \item a vertex $v\in V(G)\setminus V_i$ is $(\alpha, V_i)$-bad if $d(v,V_i)\le \alpha n$;
    \item a vertex $v\in V_i$ is $(\alpha,V_i)$-exceptional if $d(v,V_j)\le |V_j|-\alpha n$ for some $j\in [t]\setminus \{i\}$.
\end{itemize} 
\end{definition}
%For this, we introduce 
We now introduce the following definition of our desired partition.
\begin{definition}[Good partition]\label{def:good partition}
For $r,s,n\in \mathbb{N}$ with $s\le r$, and positive constants $\alpha, \beta,\beta', \gamma$, suppose that $G$ is a graph on $rn$ vertices and $\mathcal{P}=\{A_1,\ldots,A_s,B\}$ is a partition of $V(G)$. We say that $\mathcal{P}$ is an \text{$(\alpha,\beta,\beta',\gamma)$-good partition} of $G$ if all of the following hold: 
    \begin{enumerate}
[label =\rm  (A\arabic{enumi})]
    \item\label{B1} $ |A_i|\leq n+\alpha n$ for each $i\in [s]$, and $(r-s)n-r\alpha n\leq|B|\leq (r-s)n+r\alpha n$; % if $s\leq r-2$ and $n-r\alpha n\leq |B|\leq n+r\alpha n$ if $s=r-1$ and $B=\emptyset$ if $s=r$;
    \item\label{B2}  if $|A_i|>n$, then $\Delta(G[A_i])\leq \beta' n$ and the matching number of $G[A_i]$ is at least ${|A_i|-n+r}$; if $|A_i|=n$, then $G[A_i]$ is not empty;
    \item\label{B03} the number of $(\alpha^{1/5},A_i)$-good vertices is at least $|A_i|-2\alpha n$ for each $i\in[s]$; 
    \item\label{B04} for each $i\in [s]$ and each vertex $v\in V(G)\setminus A_i$, $d(v,A_i)\geq \beta n$;
    \item\label{B05} $d(v,B)\geq \beta n$ for each vertex $v\in V(G)\setminus B$ and $\de(G[B])\ge (r-s-1)n-r\alpha n$;
    \item\label{B06} the number of $(\alpha^{1/5},B)$-exceptional vertices is at most $r\alpha n$;

    %\item\label{B6} $d(v,D)\geq |D|-2\alpha^{1/5} n$ for every $(\alpha^{1/5},A_i)$-good vertex $v\in A_i$ with $A_i\neq D$;
    \item\label{B07} %\sout{$A_i$ is a $\gamma$-independent set for each $i\in [s]$, and} 
    $B$ has no $\gamma$-independent set of size at least $\frac{|B|}{r-s}$.
    %\item\label{B07} if $s=r-2$, then  for any two disjoint vertex sets of size at least $n-\beta n$ in $B$, there is at least one edge between them.
\end{enumerate}
\end{definition}

%is much easier to handle. 
Lemma~\ref{lem:construct good-partition} tells us that $G$ admits a good partition, say  $\mathcal{P}=\{A_1,\ldots,A_s,B\}$. Let $S\subseteq V(G)$ be the random subset formed by including each vertex of $V(G)$ independently with probability $\frac{1}{2}$, and let $\mathcal{P}'=\{A_1\cap S,\ldots,A_s\cap S,B\cap S\}$.
{We call $A_i\cap S$ a \textit{large} part of $\mathcal{P}'$ if $|A_i\cap S|>\frac{|S|}{r}$ and $B\cap S$ a \textit{large} part if $|B\cap S|>\frac{r-s}{r}|S|$. Otherwise the part is \textit{small}.}
Our proof proceeds in the following two steps.

\medskip
{\bf Step 1.} Prove that $\mathcal{P}'$ is a good partition of $G[S]$ with probability at least $\frac{1}{(40r^2)^r}$.
\medskip

%Based on Lemma \ref{lem:covering bad vertices for balance}, it suffices to show that $\mathcal{P}'$ is a good partition of  $G[S]$ with positive  probability. 

%We guarantee that $\mathcal{P}'$ is a good partition by employing the following two steps.

It suffices to show that $\mathcal{P}'$ satisfies~\ref{B1}-\ref{B07}. By Chernoff bound and concentration, conditions~\ref{B03}-\ref{B07} hold with high probability.
We therefore turn to the verification of~\ref{B2}, that is, every large part $A_i\cap S$ of $\mathcal{P}'$ contains a matching of size at least $|A_i\cap S|-\frac{|S|}{r}+r$. The strategy here is to restrict the size of each part  to a certain interval (make sure this happens with a constant probability), so that {$\mathcal{P}'$ satisfies~\ref{B1} and} any possible \( A_i \) with insufficient matching edges in \( G \) would inevitably turn into a small part of \(\mathcal{P}'\).

\medskip
{\bf Step 2.} Prove that $G':=G[S]$ has a  $K_r$-factor if $G'$ admits a good partition (see Lemma~\ref{lem:covering bad vertices for balance}).\medskip

In doing this, we employ Lemma~\ref{lem:balance} and for that we shall do some preliminary cleaning and contraction to leave behind an auxiliary graph together with an almost balanced partition satisfying~\ref{A1}-\ref{A2}, to which we then apply Lemma~\ref{lem:balance}. The proof of Lemma~\ref{lem:balance} is, in turn, based on a multipartite Hajnal-Szemer\'{e}di theorem (see Lemma~\ref{lem:HS thm in multipartite graph} and we will elaborate more details in Section~5).
%For the proof of Lemma~\ref{lem:covering bad vertices for balance}, we will make use of Lemma \ref{lem:balance}, which forces us to  construct an auxilary graph satisfies all of \ref{A1}-\ref{A3}. %  which  a partition satisfies all  we will make use of a result %by Gan, Han, and Hu 
%concerning $K_r$-factors in balanced $r$-partite graphs (see Lemma~\ref{lem:HS thm in multipartite graph}). 
%Under the given good partition, we give a sufficient condition to construct a balanced multi-partite graph using the matching in \textit{large} part.
%In fact, the cliques within large parts can all contribute to making the partition $\mathcal{P}'$ balanced.
% The application of Lemma~\ref{lem:balance} proceeds with the following three preliminary steps.
{To be more specific, building on a good partition $\mathcal{P}'$, we construct an auxiliary graph satisfying~\ref{A1}-\ref{A2} in  the following three steps.}
\begin{itemize}
    \item {\bf  Cleaning bad or exceptional vertices:} We first clean all vertices failing the degree condition of~\ref{A3} (i.e., bad or exceptional vertices) by using vertex-disjoint copies of $K_r$. % in $G'$. % (i.e., those failing the degree condition of \ref{A3}). 
    Note that there are only $o(n)$ such vertices, far fewer than the number of good neighbors that each such vertex has in every $A_i\cap S$. Moreover,~\ref{B05} ensures the existence of linearly many vertex-disjoint copies of $K_{r-s}$ in $G'[B\cap S]$. % within the common neighborhood of all previously selected vertices and that each vertex in $B\cap S$ lies in linearly many disjoint copies of $K_{r-s}$.
    Thus, one may find a  $K_r$-tiling $\mathcal{K}$ covering all bad or exceptional vertices.
    %This can be easily achieved by \ref{B04}-\ref{B05} as there are only $o(n)$ many such vertices, combined with the fact that each vertex is contained in linearly many disjoint copies of $K_r$. % that is orders of magnitude larger-specifically, linear in $n$.
%Denote by $\mathcal{T}$ the resulting $K_r$-tiling obtained in this step. 
Furthermore, this tiling $\mathcal{K}$ can be chosen such that $(r-s)|V(\mathcal{K})\cap A_i|=|V(\mathcal{K})\cap B|$ for each $i\in [s]$.

\item {\bf  Fixing divisibility and balancing:} Let $G'':=G'-V(\mathcal{K})$, $A_i':=(A_i\cap S)\setminus V(\mathcal{K})$ for each $i\in [s]$ and $B':=(B\cap S)\setminus V(\mathcal{K})$.
The worst-case scenario is when the size of $B'$ is not divisible by $r-s$ or too small (i.e $\frac{|B'|}{r-s}<\frac{|G''|}{r}$). To overcome this, we will separately pick up two distinct collections of $r$-cliques, denoted as  $\mathcal{H}$ and $\mathcal{R}$, %\todo{Rename $\mathcal{H}_1$ in the last proof},
such that each $K_r$ in $\mathcal{H}$  uses $r-s+1$ vertices from \(B'\); every \(K_r\) in $\mathcal{R}$ instead uses \(r-s-1\) vertices from \(B'\) and a matching edge from some large part $A_i'$. In this way, we can ensure that $(r-s)\mid |B'\setminus V(\mathcal{H}\cup\mathcal{R})|$ and moreover $b:=\frac{|B'\setminus V(\mathcal{H}\cup\mathcal{R})|}{r-s}- \frac{|G''-(\mathcal{H}\cup\mathcal{R})|}{r}\ge 0$. In doing this, we need to find a large number of disjoint copies of $K_{r-s+1}$ or $K_{r-s-1}$, and this follows from~\ref{B05} and the fact that $G'[B\cap S]$ contains no $o(1)$-independent set of size at least $\frac{|B\cap S|}{r-s}$. In fact, a supersaturation version of stability  result (see Lemma~\ref{SuperT}) implies that $G'[B\cap S]$ actually contains {$\Omega(n^{r-s+1})$ distinct} copies of $K_{r-s+1}$. Moreover, \ref{B2} ensures that each ${G'}[A_i\cap S]$ contains a matching of the required size.

    \item {\bf Contraction:} {%Recall that, for every $i \in [s]$, we have $|B\cap S| \approx (r-s)\,|A_i\cap S|$.
    To obtain an almost balanced partition as in~\ref{A1} and~\ref{A2},}  we will find a perfect tiling with $K_{r-s}$ (possibly need extra copies of $K_{r-s+1}$, denoted as $\mathcal{F}$) in $G'[B'\setminus V(\mathcal{R}\cup\mathcal{H})]$ and contract each copy of $K_{r-s}$ (resp. $K_{r-s+1}$ in $\mathcal{F}$) into a vertex (resp. an edge), yielding a new vertex set, denoted as $B^*$ whose size is close to $|A_i'|$ for every $i\in[s]$. In this case, ~\ref{A2} can be easily achieved by~\ref{B2}, and in particular, the matching number in $B^*$ can be achieved by the construction of $B^*$ and an additional constraint on $|\mathcal{F}|=(r-s)b$. Indeed, such a perfect tiling can be constructed by first taking $\mathcal{F}$ (a few number of $K_{r-s+1}$) from $B'\setminus V(\mathcal{H}\cup \mathcal{R})$ and then finding a $K_{r-s}$-factor in the remaining set in two cases depending on whether $r-s=2$ or not. For $r-s\geq 3$, a desired $K_{r-s}$-factor can be obtained from Theorem~\ref{thm:Non-extremal case r>3}; for $r-s=2$, a perfect matching can also be found by Lemma~\ref{thm:r-s=2} after a minor adjustment towards $\mathcal{K}\cup \mathcal{H}\cup \mathcal{F}$.

\end{itemize}

\medskip

\noindent{\bf Organization:} {In Section~\ref{sec:non-extremal} we consider the non-extremal case, establishing Theorem~\ref{thm:Non-extremal case}.
Section~\ref{sec:extremal} deals with the extremal case, i.e., Theorem~\ref{lem:extremal-case}.
The proof of Theorem~\ref{lem:extremal-case} splits into two parts: first we show that \(G\) admits a good partition (Lemma~\ref{lem:construct good-partition}), and then we prove that every graph possessing such a partition contains a \(K_r\)-factor (Lemma~\ref{lem:covering bad vertices for balance}).
The proofs of Lemma~\ref{lem:construct good-partition} and Lemma~\ref{lem:covering bad vertices for balance} are provided in Section~\ref{sec:lemma}. Some concluding remarks are given in the last section.}

\section{Non-extremal case}\label{sec:non-extremal}

This section is devoted to the proof of Theorem~\ref{thm:Non-extremal case}, which concerns the case when the host graph $G$ contains no large sparse set. The existence of a $K_r$-factor in such a graph $G$ was established by Gan, Han, and Hu~\cite{gan2024}, under a slightly weaker minimum degree condition.
%In this section, we prove Theorem \ref{lem:extremal-case}, which considers the case when the host graph $G$ has no large independent set. 
%\begin{definition}[$\gamma$-independent set]
%Given an $n$-vertex graph $G$ and a constant $\gamma>0$, a vertex subset $U\subseteq V(G)$ is called a $\gamma$-independent $k$-set if $e(G[U])<\gamma n^2$ and $|U|=k$.
%\end{definition}
%We divide the proof into two cases based on whether $G$ contains a $\gamma$-independent set of size $\frac{|G|}{r}$ (the \textit{extremal case}) or not (the \textit{non-extremal case}). 
%For the non-extremal case, the existence of a $K_r$-factor in $G$ was established by Gan, Han and Hu~\cite{gan2024}.

\begin{theorem}[\cite{gan2024}]\label{thm:Non-extremal case r>3}
Suppose $r\ge 3$ and $\frac{1}{n}\ll\alpha\ll\gamma\ll\frac{1}{r}$. Let $G$ be an $rn$-vertex graph with $\de(G)\ge (r-1-\alpha)n$. If $G$ has no $\gamma$-independent set of size $n$, then there is a $K_r$-factor in $G$.
\end{theorem}

Next, we introduce the Frieze-Kannan Regularity Lemma~\cite{Frieze}, which will be used in the proof of Theorem~\ref{thm:Non-extremal case}.%, which will be used to show that  w.h.p., the  bidensenessis inherited by the random induced subgraph $G[\frac{1}{2}]$. 
\begin{lemma}[\cite{Frieze}]\label{partition}
    For any $\eps>0$, there are $T,n_0>0$ such that the following holds for all $n\geq n_0$. Every $n$-vertex graph $G$ admits a vertex partition $\{V_1,\ldots,V_t\}$  with $t\leq T$ and $|V_1|\leq \cdots\leq |V_t|\leq |V_1|+1$ satisfying for any $A,B\subseteq V(G)$ we have 
    $$
    \Big|e(G[A,B])-\sum_{1\le i,j\le t}\frac{|A\cap V_i||B\cap V_j|}{|V_i||V_j|}e(G[V_i,V_j])\Big|<\eps n^2.
    $$
\end{lemma}
\medskip

Now we are ready to present a short proof of Theorem~\ref{thm:Non-extremal case}.%Now we prove the following result to end this section.
%\begin{lemma}[Non-extremal case]
%Let $\gamma>0$ and $G$ be an $((r-1)n+1)$-regular graph on $rn$ vertices. Suppose that $G$ contains no $\gamma$-independent $n$-set. Then $G[\frac{1}{2}]$ admits a $K_r$-factor with probability at least $\frac{1}{r}-o(1)$.
%\end{lemma}

\begin{proof}[{\bf Proof of Theorem~\ref{thm:Non-extremal case}}]
Choose $r\geq 3,n\in \mathbb{N}$ and a constant $\gamma$ satisfying $\frac{1}{n}\ll\gamma\ll \frac{1}{r}$. Let $G$ be an $((r-1)n+1)$-regular graph on $rn$ vertices such that $G$ contains no $\gamma$-independent set of size $n$. Let $\{V_1,\ldots,V_t\}$ be a partition of $V(G)$ obtained from Lemma~\ref{partition} applied with $\eps=\frac{1}{10}\gamma$. Consider $G[\frac{1}{2}]=G[S]$, where $S$ is a random subset of $V(G)$ with $r\mid |S|$. Note that $r \mid |S|$ is a necessary condition for $G[S]$ to contain a $K_r$-factor, which occurs with probability $\frac{1}{r}$. By Chernoff bound, {w.h.p.}~$|S|=\frac{r}{2}n\pm n^{0.6}$, $|S\cap V_i|= \frac{rn}{2t}\pm n^{0.6}$ for each $i\in [t]$,  and every vertex has degree $\frac{r-1}{2}n\pm n^{0.6}$ in $G[S]$. 
%Our goal is to prove that 
Conditioned on these events,  we prove the following claim. 
\begin{claim}\label{no-independent}
    $G[S]$ contains no $\frac{\gamma}{25}$-independent set of size $\frac{|S|}{r}$. 
\end{claim}

%By Chernoff bound, $|S\cap V_i|= \frac{rn}{2t}\pm n^{0.6}$. 
\begin{proof}
Given any $A'\subseteq S$ with $|A'|= \frac{|S|}{r}$, we are to show that $e(G[A'])\ge \frac{\gamma}{25} n^2$. 
Construct $A\subseteq V(G)$ %with $|A|\sim n$ 
so that {$|A\cap V_i|=2|A'\cap V_i|\pm 2n^{0.6}$} for all $i\in [t]$. Notice that $|A|=n\pm 3tn^{0.6}$ and $e(G[A])\geq \frac{4}{5}\gamma n^2$ as $G$ itself contains no $\gamma$-independent set of size $n$. Lemma~\ref{partition} implies that
\begin{itemize}
    \item $e(G[A])$ differs by at most $\frac{\gamma}{10}n^2$ from $\sum:=\sum_{i,j}\frac{|A\cap V_i||A\cap V_j|}{|V_i||V_j|}e(G[V_i,V_j])$;
    \item $e(G[A'])$ differs by at most $\frac{\gamma}{10}n^2$ from $\sum_{i,j}\frac{|A'\cap V_i||A'\cap V_j|}{|V_i||V_j|}e(G[V_i,V_j])\geq {\frac{1}{5}}\sum$. %{(make it more precise!)}.
\end{itemize}
{{It follows that 
$$
e(G[A'])\geq \frac{1}{5}\sum-\frac{1}{10}\gamma n^2\geq \frac{1}{5}\Big(e(G[A])-\frac{\gamma}{10} n^2\Big)-\frac{\gamma}{10}n^2\geq \frac{7}{50}\gamma n^2-\frac{\gamma}{10} n^2= \frac{\gamma}{25}n^2. 
$$}} %{(How? add more details)}. 
That is, $G[S]$ contains no $\frac{\gamma}{25}$-independent set of size $\frac{|S|}{r}$. 
\end{proof}
Combined with the condition $\de(G[S])\geq \frac{r-1}{2}n- n^{0.6}$, Theorem~\ref{thm:Non-extremal case r>3} implies that $G[S]$ contains a $K_r$-factor, which happens with probability at least $\frac{1}{r}-o(1)$. 
%By Chernoff bound, $|A\cap S|= \frac{n}{2}\pm n^{0.6}$ and $e(G[A\cap S])=\frac{\gamma}{4}n^2\pm n$. 
\end{proof}

\section{Extremal case}\label{sec:extremal}

For the extremal case, our first task is to establish a good partition of the graph $G$, provided that it contains a large sparse set. 
\begin{lemma}\label{lem:construct good-partition}
    Let $r,n$ be positive integers with $r\ge 3$, and choose $\frac{1}{n}\ll \gamma\ll\alpha\ll\beta\ll\gamma'\ll \frac{1}{r}.$ {Let $G$ be an $rn$-vertex graph with $\delta(G)\geq (r-1)n+1$.} Suppose that $G$ contains a $\gamma$-independent set of size $n$. Then there exists an $(\alpha,\beta,\beta,\gamma')$-good partition $\mathcal{P}=\{A_1,\cdots,A_s,B\}$ of $G$ for some $s\in [r]$. %Furthermore, we have that

    %\begin{enumerate}[label =\rm  (B2')]
        %\item\label{B2'}  $\Delta(G[A_i])\leq \beta n$ and the matching number of $G[A_i]$ is at least $\frac{|A_i|-n}{4\beta}$ if $|A_i|>n$; 
        %\end{enumerate}
        %\begin{enumerate}[label =\rm  (B6')]
        %\item\label{B6'}  the number of $(\frac{\beta}{2},B)$-exceptional vertices is at most $r\alpha n$;
        %\end{enumerate}
        %\begin{enumerate}[label =\rm  (B7')]
%\item\label{b7'}   $d(v,D)\geq |D|-\alpha^{1/5} n$ for every $(\alpha^{1/5},A_i)$-good vertex $v\in A_i$ with $A_i\neq D$.
    %\end{enumerate}
%\begin{enumerate}[label =\rm  (B7')]
    %\item\label{B7'} $A_i$ is a $\gamma'$-independent set for each $i\in [s]$, and $B$ has no $\gamma'$-independent set of size $n$.
    %\end{enumerate}
\end{lemma}

Given a good partition of $G$, one can expect that $G[S]$ also admits a good partition. If so, then the following lemma would ensure the existence of a $K_r$-factor in $G[S]$.

\begin{lemma}\label{lem:covering bad vertices for balance}
Let $r,n$ be positive integers with $r\ge 3$ and choose $\frac{1}{n}\ll\alpha\ll\beta',  \beta\ll\gamma\ll \frac{1}{r}$. Let $G$ be an $rn$-vertex graph with {$\de(G)\geq (r-1)n-\alpha n$.}  %\sout{$(r-1)n-\alpha n\leq \de(G)\leq \Delta(G) \leq (r-1)n+\alpha n$}. 
If $G$ admits an $(\alpha,\beta,{\beta'},
\gamma)$-good partition, then $G$ has a $K_r$-factor. % partition $\mathcal{P}$ can be transformed into a balanced one by removing a $K_r$-tiling of seize at most ?? vertices.
\end{lemma}
We will prove Theorem~\ref{lem:extremal-case} using Lemmas~\ref{lem:construct good-partition} and~\ref{lem:covering bad vertices for balance}, whose proofs are postponed to the next section.
\subsection{Proof of Theorem~\ref{lem:extremal-case}}
In this subsection we prove Theorem~\ref{lem:extremal-case}, which establishes Theorem~\ref{thm:the random version of main thm} for graphs containing a sparse set of size $n$. 
%In this section we prove our main result Lemma~\ref{lem:extremal-case} using the sufficient conditions for $K_r$-factors in Lemma~\ref{lem:covering bad vertices for balance}. 
As the size of a random subset of $A$ is distributed as the binomial $B(|A|, 1/2)$, we shall use the following normal approximation, which is a consequence of the Central Limit Theorem.
%As the size of a random subset of a set $A$ has the binomial $B(|A|, 1/2)$ distribution, we need the following observation on the difference of independent binomials.

\begin{lemma}[\cite{Keevash2025}]\label{lem-int}
If $X\sim B(n,1/2)$, then $\mathbb{P}(a\sqrt{n}/2\leq X-n/2\leq b\sqrt{n}/2)=\int_a^b \frac{1}{\sqrt{2\pi}}e^{-t^2/2}\,dt+o(1)$.
\end{lemma}
We will also use the following version of Chernoff bound.
\begin{lemma}[Chernoff bound]
    Let $X$ be a binomial random variable and $\delta>0$. Then 
    $$
    \mathbb{P}[|X-\mathbb{E}X|\geq \delta \mathbb{E}X]\leq 2\exp(-\delta^2\mathbb{E}X/(2+\delta)).
    $$
\end{lemma}
The following lemma ensures that for any  $rn$-vertex $((r-1)n+1)$-regular graph $G$ with a balanced partition  $\{A_1,A_2,\ldots,A_r\}$ (i.e., $|A_1|=\cdots=|A_r|=n$), there exists a large matching within some $G[A_i]$. The proof will be given in the next section.  % a subgraph induced by a set of size $n$. %  gives that for a balanced partition, there is a matching in some part, which uses regularity.

\begin{lemma}\label{vertex-cover}
Let $G$ be an $((r-1)n+1)$-regular graph on $rn$ vertices and $\{A_1,A_2,\ldots,A_r\}$ be a balanced partition of $V(G)$. Suppose that $C_i$ is a minimum vertex cover of $G[A_i]$ for each $i\in [r]$. Then $\max\{|C_i|:i\in [r]\}\geq \frac{1}{r}\sqrt{n}$.
\end{lemma}
%The following simple remark  will be used in our proof.

\begin{rmk}\label{remark:vertex-cover}
Let $M$ be a maximal matching in a graph $H$. Then $V(M)$ is a vertex cover in $H$. In particular, $|M| \geq |C|/2$ where $C$ is a minimum vertex cover {in $H$}. 
If moreover $H$ has maximum degree $\Delta$,  %\sout{ and $C$ is a vertex cover in $H$}. 
then $e(H) \leq \Delta |C|$.
%Let $M_i$ be a maximum matching of $A_i$. By Vizing's theorem, we have
%\begin{align*}
%|M_i|\ge \frac{e(G[A_i])}{\Delta(G[A_i])+1}\ge \frac{|A_i|\cdot (k_i+1)}{2(\beta n+1)}\ge \frac{k_i}{2\beta}.
%\end{align*}
\end{rmk}

We conclude this section by giving the proof of Theorem~\ref{lem:extremal-case}.
\begin{proof}[{\bf Proof of Theorem~\ref{lem:extremal-case}}]
Choose $r\geq 3,n\in \mathbb{N}$ and constants satisfying $$\frac{1}{
n}\ll\gamma\ll \alpha\ll \beta\ll \delta\ll\gamma'\ll \frac{1}{r}.
$$
Assume that $G$ is an $((r-1)n+1)$-regular graph on $rn$ vertices and $G$ contains a $\gamma$-independent set of size $n$.
By Lemma~\ref{lem:construct good-partition}, $G$ admits  an $(\alpha,\beta,\beta,\gamma')$-good partition $\mathcal{P}:=\{A_1,\ldots,A_s,B\}$ for some $s\in [r]$, that is, $\mathcal{P}$ satisfies~\ref{B1}-\ref{B07}. %partition $\mathcal{P}:=\{A_1,\cdots,A_s,B\}$ of $V(G)$ satisfying~\ref{B1}-\ref{B07}.

Consider the random subset $S\subseteq V(G)$ obtained by including each vertex of $V(G)$ independently with probability $\frac{1}{2}$. Then, w.h.p.~we have {$\de(G[S])\geq \frac{(r-1)n}{2}-n^{0.6}$}. %By Chernoff bound, {w.h.p.}~$|S|=\frac{r}{2}n\pm n^{0.6}$, $|S\cap A_i|= \frac{n}{2}\pm r\alpha n$ for each $i\in [s]$ and $|S\cap B|= \frac{(r-s)n}{2}\pm r\alpha n$. 
Let $\mathcal{P}'=\{A_1\cap S,\ldots,A_s\cap S,B\cap S\}$. By Lemma~\ref{lem:covering bad vertices for balance}, in order to obtain a $K_r$-factor in $G':=G[S]$, it suffices to determine the probability that  $\mathcal{P}'$ is a  $(3\alpha,\frac{\beta}{4},3\beta,\frac{\gamma'}{25})$-good partition of $G'$ under the condition that $r\mid |S|$.  
%In what follows, we condition on the event that  \( r \mid |S| \), which occurs with probability \( \frac{1}{r} \) and is necessary for a \(K_r\)-factor in \(G[S]\).
%Note that $S\sim B(rn,1/2)$ and $|S|=\frac{rn}{2}\pm n^{0.6}$ with high probability.
%By Chernoff  bound,  %~$|S|=\frac{rn}{2}\pm n^{0.1}$, $d_{G[S]}(v)=(r-1)\frac{|S|}{r}\pm n^{0.1}$ for each $v\in S$. %, and %$|S\cap A_i|=\frac{|A_i|}{2}\pm n^{0.1}$ for each $i\in [s]$, and $|S\cap B|=\frac{|B|}{2}\pm n^{0.1}$. 
%Moreover, by chernoff bound, 
%we have the following: 
%\begin{claim}\label{claim:B1-B06 in S}
%Given constants $\frac{1}{n}\ll\gamma\ll\alpha\ll\beta\ll1$, there exists a constant $\gamma'\ll\frac{1}{r}$ such that the random 

\begin{claim}\label{claim-prob}
With high probability, $\mathcal{P}'$ satisfies the following properties:
\begin{enumerate}[label=\rmfamily (A\arabic*), start=3]
    %\item[{\rm (B1)}]\label{D1} $\frac{n}{2}-2r\alpha n\leq |A_i|\leq \frac{n}{2}+2r\alpha n$ for each $i\in [s]$ and $(r-s)\frac{n}{2}-2r\alpha n\leq |B|\leq (r-s)\frac{n}{2}+2r\alpha n$; %\item\label{D1}
    %$|A_i\cap S|=\frac{|A_i|}{2}\pm n^{0.1}$ for each $i\in [s]$ and $|B\cap S|=\frac{|B|}{2}\pm n^{0.1}$;%$|A_i\cap S|=\frac{n}{2}+ n^{0.1}$ if $|A_i|\le n$ and $|A_i\cap S|=\frac{|A_i|}{2}|\pm n^{0.1}$ if $|A_i|=n+k_i$ where $k_i>0$;
    %\item[{\rm (B2)}]\label{D2} $\Delta(G[A_i\cap S])\leq \beta n$ for each $i\in [s]$;%\todo{relabel $D_i$ as $B_i$}
    \item\label{D1} the number of $((3\alpha)^{1/5},A_i\cap S)$-good vertices is at least $|A_i\cap S|-{6\alpha \frac{|S|}{r}}$ for each $i\in [s]$; 
    \item\label{D3} $d(v,A_i\cap S)\geq {\frac{\beta}{4}\frac{|S|}{r}}$ for each vertex $v\in S\setminus A_i$ with $i\in [s]$;
    \item\label{D4} $d(v,B\cap S)\geq {\frac{\beta}{4}\frac{|S|}{r}}$ for each vertex $v\in S\setminus B$ and $\de(G[B\cap S])\ge {\frac{(r-s-1)|S|}{r}-3\alpha |S|}$;
    \item\label{D5} the number of $((3\alpha)^{1/5},B\cap S)$-exceptional vertices is at most ${3\alpha |S|}$;
    %\item\label{D6} for any two sets of size $\frac{|S|}{r}$ in $B\cap S$ there are at least $\frac{\beta n}{4}$ edges between them;
    %\item\label{D7} $d(v,D\cap S)\ge \frac{|D\cap S|}{2}-2\alpha^{1/5}n$ for every $(\alpha^{1/5},A_i)$-good vertex $v\in A_i\cap S$ where $A_i\cap S\neq D$ and $D\in \mathcal{P}'$;
    \item\label{D7} %\sout{$A_i\cap S$ is a $5\gamma_1$-independent set for each $i\in [s]$ and} 
    $B\cap S$ has no $\frac{\gamma'}{25}$-independent set of size  at least $\frac{|B\cap S|}{r-s}$.
    %\item[{\rm (B06)}]\label{D8} if $s=r-2$, then  for any two sets of size $n-\beta n$ in $B$, there is at least one edge  between them.
\end{enumerate}
\end{claim}

\begin{proof}%[Proof of Claim \ref{claim-prob}]
Note that $|S|\sim B(rn,1/2)$, $|A_i\cap S|\sim B(|A_i|,1/2)$ for all $i\in[s]$ {and $|B\cap S|\sim B(|B|,1/2)$}. By Chernoff bound, w.h.p.~$|S|=\frac{rn}{2}\pm n^{0.6}$, $|A_i\cap S|=\frac{|A_i|}{2}\pm n^{0.6}$ for each $i\in [s]$ and $|B\cap S|=\frac{|B|}{2}\pm n^{0.6}$.

(A3) %{Simplify this. It seems very trivial}. 
Let $v\in A_i$ be an $(\alpha^{1/5},A_i)$-good vertex in $G$ for some $i\in [s]$.  Recall that $d(v,A_i)\leq \alpha^{1/5} n$. %If $d(v,A_i)\leq (3\alpha)^{1/5}\frac{n}{3}$, then w.h.p. we have $$d(v,A_i\cap S)\leq d(v,A_i)\leq (3\alpha)^{1/5}\frac{|S|}{r}.$$  If $d(v,A_i)\geq (3\alpha)^{1/5}\frac{n}{3}$, t
Then by Chernoff bound, w.h.p.~we have $$d(v,A_i\cap S)\leq \frac{1}{2}\alpha^{1/5}n+ n^{0.6}\leq (3\alpha)^{1/5}\frac{|S|}{r}.$$ %Hence $d(v,A_i\cap S)\leq d(v,A_i)\leq \alpha^{1/5} n\leq (2^6\alpha)^{1/5}\frac{|S|}{r}$. 
Together with~\ref{B03}, there are at least $|A_i\cap S|-{2}\alpha n{\ge |A_i\cap S|-6\alpha\frac{|S|}{r}}$ vertices in $G'$ that are $((3\alpha)^{1/5},A_i\cap S)$-good.\medskip

(A4)-(A5) trivially follow from the Chernoff bound. \medskip

(A6) Let $v\in B$ be a vertex that is not $(\alpha^{1/5},B)$-exceptional in $G$. 
Then $d(v,A_j) > |A_j| - \alpha^{1/5} n$ for every $j \in [s]$. 
By Chernoff bound, w.h.p. we have
\[
d(v,A_j \cap S)\geq  \frac{|A_j| - \alpha^{1/5} n}{2} - n^{0.6} \ge |A_j \cap S| - (3\alpha)^{1/5}\frac{|S|}{r} \ \text{for all } j \in [s].
\]
It follows from~\ref{B06} that the number of $((3\alpha)^{1/5},B \cap S)$-exceptional vertices in $G'$ is at most $r\alpha n\le {3\alpha |S|}$. \medskip%\todo{maybe we can replace this $2\alpha$ with $3\alpha$}. 

(A7) %\sout{By \ref{B07}, w.h.p. we have that for each $i\in [s]$,}  $$e(G[A_i\cap S])\leq e(G[A])\leq \gamma_1 n^2\leq 5\gamma_1 \Big(\frac{|S|}{r}\Big)^2.$$ Thus, $A_i\cap S$ is a $5\gamma_1$-independent set of $G'$. 
By applying Claim~\ref{no-independent} with $(G[B],\gamma')$ in place of $(G,\gamma)$, we obtain that $G[B\cap S]$ contains no $\frac{\gamma'}{25}$-independent set of size at least $\frac{|B\cap S|}{r-s}$, as desired. 
%Notice that  each edge in $G$ is selected into $G[S]$ with probability $\frac{1}{4}$. 
%and for any set 
%Let $B'$ be an arbitrary subset of $B\cap S$ with size  $\frac{|B\cap S|}{r-s}$. Then there exists a vertex subset $B''$ of $B$ with size $\frac{|B|}{r-s}\pm n^{0.6}$ such that $B'\subseteq B''$.  there exists an \(n\)-vertex set \(B''\supseteq B'\) such that \todo{How to choose $B''$?, The following is too informal to believe.}
%$$e(G[B']){\geq ??}{\Big(\frac{1}{4}-{o(1)??}\Big)e(G'[B''])\geq \Big(\frac{1}{4}-o(1)\Big)\gamma_2n^2\geq \frac{\gamma_2}{2}\Big(\frac{|S|}{r}\Big)^2.$$ 
%Thus, w.h.p.~(B7) holds. {(Here we refer to the proof Thm2.1)}. 
\end{proof}
%Since $A_i$ is a $\gamma'$-independent set in $G$, we have $e(G[A_i\cap S])\leq (\frac{1}{4}+o(1))e(G[A])\leq (\frac{1}{4}+o(1))\gamma'n^2\leq 2\gamma'(|S|/r)^2$. Since $B$ contains no is a $\gamma'$-independent $n$-set in $G$, for any set $B'\subseteq B\cap S$ with size $\frac{n}{2}$, we have $e(G'[B'])\geq (\frac{1}{4}-o(1))\gamma'n^2\geq \frac{1}{2}\gamma'(|S|/r)^2$. %together with \ref{B7'}, it is easy to check that $\mathcal{P}'$ w.h.p.~satisfies~\ref{B07}.

%Since $\mathcal{P}$ satisfies~\ref{B03}, that is, there are at least $|A_i|-2\alpha n$ $(\alpha^{1/5},A_i)$-good vertices in each $A_i$, it follows that w.h.p.~the number of $(\frac{\alpha^{1/5}}{2},A_i\cap S)$-good vertices is at least $|A_i\cap S|-4\alpha n$. By the similar argument, we deduce that $\mathcal{P}'$ w.h.p.~satisfies~\ref{B03}-\ref{B07}. As each edge in $G$ is selected into $G[S]$ with probability $1/4$, it is easy to check that $\mathcal{P}'$ w.h.p.~satisfies~\ref{B07}.

Let $\mathbf{E}_0$ be the event that~\ref{D1}-\ref{D7} hold, and let $\mathbf{E}$ be the event that $\mathcal{P}'$ is a $(3\alpha,\frac{\beta}{4},3\beta,\frac{\gamma'}{25})$-good partition of $G'$. 
By Claim~\ref{claim-prob}, we have $\mathbb{P}[\mathbf{E}_0]\geq 1-\frac{1}{n}$. %\footnote{Make it more precise, for example $1-1/n$}.
To compute $\mathbb{P}[\mathbf{E}]$, it remains to estimate {the probability that~\ref{B1} and~\ref{B2} hold}. % conditioned on $\mathbf{E}_0$. 

Next, we will restrict the size of each part so that they fall into a prescribed interval with a positive probability.
This ensures that every large part \( A_i \cap S \) of \(\mathcal{P}'\) has a matching of size at least \( |A_i \cap S| - \frac{|S|}{r} + r \). Define $k_i:=|A_i|-n$ for each $i\in [s]$.  
% We decompose the event $\mathbf{E}$ into $s+1$ sub-events $\mathbf{E_i}$\todo{put this to the where it is defined?}, each of which measures the size of $A_i\cap S$ for each $i\in [s]$ or $B\cap S$ for $i=s+1$. Note that $\mathbb{P}[\mathbf{E}]=\prod_{i=1}^{s+1} \mathbb{P}[\mathbf{E_i}]$.
We proceed by considering the following  three cases based on the value of $\max_{i\in [s]}\{k_i\}$.

\medskip
{\bf Case 1. $\max_{i\in [s]}\{k_i\}>\delta \sqrt{n}$.}% and $r-s\neq 1$.}
\medskip

Without loss of generality, assume that $k_1:=\max_{i\in [s]}\{k_i\}$. By Lemma~\ref{lem-int}, %and Lemma~\ref{lem:binomial distribution}, 
we have 
\[
\mathbb{P}\Big[\mathbf{E_1}:\frac{|A_1|}{2}+{\sqrt{|A_1|}}\leq |A_1\cap S|\leq \frac{|A_1|}{2}+5{\sqrt{|A_1|}}\Big]
\footnote{which differs by at most $\delta$ from $$
\int_{2}^{10}\frac{1}{\sqrt{2\pi}}e^{-t^2/2}\,dt=\Phi(10)-\Phi(2)\geq \frac{1}{50},
$$ 
where we take Taylor expansion {$\Phi(z)=\frac{1}{2}+\frac{1}{\sqrt{2\pi}}\sum_{k=0}^{\infty}\frac{(-1)^kz^{2k+1}}{k!\cdot 2^k\cdot(2k+1)}$}.}\ge \frac{1}{100}.\]
%Thus, $\mathbb{P}(E)\geq \frac{1}{2^r}-o(1)$. 
Similarly, for each $i\in [2,s]$ we 
have 
$$
\mathbb{P}\Big[\mathbf{E}_i:\frac{|A_i|}{2}-\frac{1}{2r}{\sqrt{|A_i|}}\leq |A_i\cap S|\leq \frac{|A_i|}{2}+\frac{1}{2r}{\sqrt{|A_i|}}\Big]\geq \frac{1}{2r},
$$ 
and
\[
\mathbb{P}\Big[\mathbf{E}_{s+1}:\frac{|B|}{2}-\frac{1}{2r}{\sqrt{|B|}}\leq |B\cap S|\leq \frac{|B|}{2}+\frac{1}{2r}{\sqrt{|B|}}\Big]\geq \frac{1}{2r}.
\]
Therefore, conditioned on all events $\mathbf{E}_i$ for $i\in [s+1]$, % when all these events occur simultaneously, 
we apply the Cauchy-Schwarz inequality to obtain   
\begin{align*}
|S|&=\sum_{i\in [s]}|A_i\cap S|+|B\cap S|\geq \frac{rn}{2}+{\sqrt{|A_1|}-\frac{1}{2r}\sum_{i\in [2,s]}\sqrt{|A_i|}-\frac{1}{2r}\sqrt{|B|}}\\
&=\frac{rn}{2}+{\frac{1}{2r}\Big((2r+1)\sqrt{|A_1|}-\sum_{i\in [s]}\sqrt{|A_i|}-\sqrt{|B|}\Big)}\\
&\geq 
\frac{rn}{2}+{\frac{1}{2r}\Big((2r+1)\sqrt{|A_1|}-r\sqrt{n}\Big)}\\
&\geq\frac{rn}{2}+\frac{r+1}{2r}\sqrt{|A_1|},  
\end{align*}
and 
$$
|S|=\sum_{i\in [s]}|A_i\cap S|+|B\cap S|\leq \frac{rn}{2}+5\sqrt{|A_1|}+\frac{r-1}{2r}\sqrt{|A_1|}\le\frac{rn}{2}+6\sqrt{|A_1|}.
$$
Hence, 
\begin{align}\label{A1capS}
|A_1\cap S|-\frac{|S|}{r}\le\frac{|A_1|}{2}+5\sqrt{|A_1|}-\frac{n}{2}-\frac{r+1}{2r^2}\sqrt{|A_1|}\le\frac{|A_1|-n}{2}+5\sqrt{|A_1|}\leq 3\alpha \frac{|S|}{r};
\end{align}
for each $i\in [2,s]$, we have 
\begin{align}\label{AcapS}
|A_i\cap S|-\frac{|S|}{r}\le\frac{|A_i|}{2}+\frac{1}{2r}\sqrt{|A_i|}-\frac{n}{2}-\frac{r+1}{2r^2}\sqrt{|A_1|}\leq \frac{|A_i|-n}{2}-\frac{1}{2r^2}\sqrt{|A_1|}\leq 3\alpha \frac{|S|}{r}.
\end{align}
%and 
%\begin{align*}
%{\sout{|A_1\cap S|-\frac{|S|}{r}\geq |A_i\cap S|-\frac{|S|}{r}\ge\frac{|A_i|}{2}-\frac{1}{2r}\sqrt{n}-\frac{n}{2}-\frac{6}{r}\sqrt{n}=\frac{|A_i|-n}{2}-\frac{13}{2r}\sqrt{n}\geq -3\alpha |S|.}}
%\end{align*}
%\begin{align}\label{AcapS}
%|A_i\cap S|-\frac{|S|}{r}\le\frac{|A_i|}{2}+\frac{1}{2r}\sqrt{n}-\frac{n}{2}-\frac{r+1}{2r^2}\sqrt{n}=\frac{|A_i|-n}{2}-\frac{1}{2r^2}\sqrt{n}\leq 3\alpha |S|,
%\end{align}
%and 
%\[
%|A_i\cap S|-\frac{|S|}{r}\ge\frac{|A_i|}{2}-\frac{1}{2r}\sqrt{n}-\frac{n}{2}-\frac{6}{r}\sqrt{n}=\frac{|A_i|-n}{2}-\frac{13}{2r}\sqrt{n}\geq -3\alpha |S|.
%\]
Similarly, 
\[
-3\alpha \frac{|S|}{r}\leq |B\cap S|-{(r-s)\frac{|S|}{r}}<3\alpha \frac{|S|}{r}.
\]
That is,~\ref{B1} holds. Fix an $i\in [s]$ with $|A_i\cap S|\geq \frac{|S|}{r}$. If \(i=1\), then the assumption \(\max_{i\in[s]}\{k_i\} > \delta\sqrt{n}\) yields \(|A_1|\ge n+\delta\sqrt{n}\);  if \(i\in[2,s]\), then by~\eqref{AcapS} we have \(|A_i|\ge n+\frac{1}{r^2}\sqrt{n}\). %$A_i\cap S$ is a large part of $\mathcal{P}'$, then by~\eqref{AcapS} one has $|A_i|\geq n+\frac{1}{r^2}\sqrt{n}$. 
Recall that $\Delta(G[A_i])\leq \beta n$. Then, w.h.p.~we have $\Delta(G[A_i\cap S])\leq \Delta(G[A_i])\leq \beta n\leq 3\beta \frac{|S|}{r}$.  By Remark~\ref{remark:vertex-cover} and {the condition that} $G$ is $((r-1)n+1)$-regular, we conclude that the matching number of $G[A_i]$ is at least $$\frac{e(G[A_i])}{2\beta n}\geq \frac{\delta(G[A_i])|A_i|}{4\beta n}\geq \frac{(|A_i|-n+1)|A_i|}{4\beta n}\geq \frac{|A_i|-n+1}{5\beta}.$$ 
%Moreover, \ref{B1} implies that $|B\cap S|\leq \frac{(r-s)|S|}{r}$ if $r-s\geq 2$. 
%Otherwise, $|A_i\cap S|\le \frac{|S|}{r}$.
%For the former case, b
By Chernoff bound, {\eqref{A1capS}-\eqref{AcapS} and $\beta\ll \delta$}, w.h.p.~we have that $G[A_i\cap S]$ contains a matching of size at least $\frac{|A_i|-n}{21\beta}>|A_i\cap S|-\frac{|S|}{r}+r$, thus~\ref{B2} holds.
It follows that  $${\mathbb{P}[\mathbf{E}]\geq \mathbb{P}\big[\mathbf{E}_0\mathbf{E}_1\mathbf{E}_2\cdots\mathbf{E}_{s+1} \big]\geq\mathbb{P}\Big[{\bigcap_{i\in [s+1]}\mathbf{E}_i }\Big]+%\footnote{are they independent events under $E_0$?}
\mathbb{P}[\mathbf{E}_0 ]-1\ge \frac{1}{100(2r)^{r-1}}-\frac{1}{n}.}
$$ 
%{more computational details}.  {what about $B04-B06$?} 
% and $G[B\cap S]$ contains a matching of size at least $|B\cap S|-\frac{|S|}{r}$ if $r-s=1$.
%Notice that $\mathbb{P}(r\mid |S|)=\frac{1}{r}$. Together with Lemma~\ref{lem:covering bad vertices for balance}, we obtain that $G[S]$ has a $K_r$-factor with probability at least $\frac{1}{100r(2r)^{r-1}}$. 

% {\color{red}Note that $k_1\le k_2\le\cdots \le k_s$ and $k_s\ge \de n^{0.2}$. Together with $|S|=\frac{rn}{2}\pm n^{0.1}$. Then $\frac{|S|}{r}\ge \frac{n}{2}+\de n^{0.1}$. Thus, Claim~\ref{claim:B1-B06 in S}\ref{D1} gives that if $A_i$ is a small part or $k_i\le\de n^{0.2}$ with respect to $\mathcal{P}$, then w.h.p.~$A_i\cap S$ is also a small part with respect to $\mathcal{P}'$. If $k_i>\de n^{0.2}$, Lemma~\ref{lem:inherit matching} implies that there is a matching of size at least $\frac{k_i}{32\beta}$ in $A_i\cap S$ with probability $1-o(1)$. Therefore, we verify that condition~\ref{B2} in Lemma~\ref{good-partition} holds. Thus, we have
% \[
% P[G[\frac{1}{2}]~\text{contains a $K_r$-factor}]\ge P[r\mid |S|]-o(1)\ge \frac{1}{r}-o(1).
% \]}

\medskip
{\bf Case 2.} 
$\max_{i\in [s]}\{k_i\}\leq \de \sqrt{n}$ and $s=r$.
\medskip

%For convenience, we denote $B$ as $A_r$ when $r-s=1$.
We consider any balanced partition $\{A_1^*,A_2^*,\ldots,A_r^*\}$ obtained from $\{A_1,A_2,\ldots,A_r\}$ by moving less than $\de \sqrt{n}$ vertices from each large part to small parts. Let $C_i$ be a minimum vertex cover of $G[A_i^*]$ for each $i\in [r]$. By Lemma~\ref{vertex-cover}, we have $\max\{|C_1|,\ldots,|C_r|\}\geq \frac{1}{r}\sqrt{n}$. Without loss of generality, assume that $|C_1|\geq \frac{1}{r}\sqrt{n}$. 
As in Remark~\ref{remark:vertex-cover}, $G[A_1^*]$ has a matching $M$ of size at least $\frac{\sqrt{n}}{2r}$. By Chernoff bound, w.h.p.~$M[S]$ contains at least $\frac{2\sqrt{n}}{17r}$ edges. Observe that at most $\de \sqrt{n}$ edges in $M[S]$ contain a {vertex that is moved from a large part to $A_1$}, so this gives a matching of size at least $\frac{\sqrt{n}}{9r}$ in ${G'}[A_1\cap S]$. By Lemma~\ref{lem-int}, % and Lemma~\ref{lem:binomial distribution},
we have 
\begin{align*}\label{eq:the size of A1}
&\ \mathbb{P}\Big[\mathbf{E}_1:\frac{|A_1|}{2}+ \frac{1}{20r}\sqrt{{|A_1|}}\leq |A_1\cap S|\leq \frac{|A_1|}{2}+\frac{1}{10r}\sqrt{{|A_1|}}\Big]\\
=&\ \mathbb{P}\Big[\frac{|A_1|}{2}+\frac{1}{20r}\sqrt{{|A_1|}}\leq B(|A_1|,1/2)\leq \frac{|A_1|}{2}+\frac{1}{10r}\sqrt{{|A_1|}}\Big]\\
%&\ge \int_{\frac{1}{10r}}^{\frac{1}{5r}}\frac{1}{\sqrt{2\pi}}e^{-t^2/2}\,dt-\de\\
\geq&\  \Phi\Big(\frac{1}{5r}\Big)-\Phi\Big(\frac{1}{10r}\Big)-\de\geq \frac{1}{30r}.
\end{align*}
Similarly, for each $i\in [2,r]$ we 
have 
\begin{align*}
\mathbb{P}\Big[\mathbf{E}_i:\frac{|A_i|}{2}-\frac{1}{40r^2}\sqrt{{|A_i|}}\leq |A_i\cap S|\leq \frac{|A_i|}{2}+\frac{1}{40r^2}\sqrt{{|A_i|}}\Big]\geq \frac{1}{40r^2}.
\end{align*}
Therefore, conditioned on all events $\mathbf{E}_i$ for $i\in [r]$, by the Cauchy-Schwarz inequality we have  
\begin{align*}
|S|&=\sum_{i\in [r]}|A_i\cap S|\ge \frac{rn}{2}+\frac{1}{20r}\sqrt{{|A_1|}}-\frac{1}{40r^2}\sum_{i\in [2,r]}\sqrt{|A_i|}\\
&= \frac{rn}{2}+\frac{1}{40r^2}\Big((2r+1)\sqrt{|A_1|}-\sum_{i\in [r]}\sqrt{|A_i|}\Big)\\
&\ge \frac{rn}{2}+\frac{1}{40r^2}\left((2r+1)\sqrt{|A_1|}-r\sqrt{n}\right).
\end{align*}
% \leq \frac{rn}{2}+\frac{5r-1}{40r^2}\sqrt{n}
Note that $\max_{i\in [s]}\{k_i\}\leq \de \sqrt{n}$, {which implies that $|A_i|=n\pm r\de\sqrt{n}$ for each $i\in [r]$. As $\de\ll\frac{1}{r}$,} for each $i\in [2,r]$, we have 
\begin{align*}
|A_i\cap S|-\frac{|S|}{r}&\le\frac{|A_i|}{2}+\frac{1}{40r^2}\sqrt{|A_i|}-\frac{n}{2}-\frac{1}{40r^3}\left((2r+1)\sqrt{|A_1|}-r\sqrt{n}\right)\\
&\leq \frac{\delta\sqrt{n}}{2}-\frac{1}{40r^3}\left(2r\sqrt{|A_1|}-r\sqrt{n}\right)<0.
\end{align*}
%and 
%\[{\sout{|A_i\cap S|-\frac{|S|}{r}\ge\frac{|A_i|}{2}-\frac{1}{40r^2}\sqrt{n}-\frac{n}{2}-\frac{5r-1}{40r^3}\sqrt{n}=\frac{|A_i|-n}{2}-\frac{6r-1}{40r^3}\sqrt{n}\geq -3\alpha |S|.}}
%\]
Similarly, 
$$
%-3\alpha |S|\le{\frac{|A_1|-n}{2}+\frac{2r-5}{40r^2}\sqrt{n}}\leq 
|A_1\cap S|-\frac{|S|}{r}<\frac{\sqrt{|A_1|}}{10r}<\frac{2\sqrt{n}}{19r}<3\alpha \frac{|S|}{r}. 
$$
This implies that~\ref{B1} holds and $A_i\cap S$ is a small part of $\mathcal{P}'$ for each $i\in [2,r]$. 
Recall that $G[A_1\cap S]$ contains a matching of size at least $\frac{\sqrt{n}}{9r}$, 
which is larger than $|A_1\cap S|-\frac{|S|}{r}+r$. Thus~\ref{B2} holds. It follows that  $${\mathbb{P}[\mathbf{E}]\geq \mathbb{P}\big[\mathbf{E}_0\mathbf{E}_1\mathbf{E}_2\cdots\mathbf{E}_r \big]\geq \mathbb{P}\Big[{\bigcap_{i\in [r]}\mathbf{E}_i}\Big]+
\mathbb{P}[\mathbf{E}_0 ]-1\ge  \frac{1}{30r(40r^2)^{r-1}}-\frac{1}{n}.}$$
 %Together with Lemma~\ref{lem:covering bad vertices for balance}, we know that $G[S]$ has a $K_r$-factor with probability at least $\frac{1}{(40r^2)^r}$, as desired.

\medskip
{\bf Case 3.} $\max_{i\in [s]}\{k_i\}\leq \de\sqrt{n}$ and $s\leq r-1$.
\medskip

%In this case, $B$ is a set of size $(r-s)n$ containing no $\gamma'$-independent $n$-set. Choose $\eps'\ll\eps\ll \gamma'$.  If $r-s=1$, then there are  $\eps n$ vertex disjoint edges in $G[B]$ as $e(G[B])\ge \gamma'n^2$ and $\gamma'\gg \eps$. If $r-s\ge 2$, then \ref{B05} implies  
%\[
%\de(G[B])\ge (r-1)n+1-sn-r\alpha n\ge \left(1-\frac{1}{r-s}-\beta\right)\cdot|B|.
%\]
%By Theorem \ref{SuperT}, there are at least $\eps n$ vertex-disjoint copies of $K_{r-s+1}$ in $G[B]$. By Chernoff bound, we know that w.h.p.~$G[B\cap S]$ has at least $\frac{\eps}{2^{r+1}}n$ vertex disjoint  copies of $K_{r-s+1}$. 
By a similar discussion as in Case~2, for each $i\in [s]$ we have 
\[
\mathbb{P}\Big[\mathbf{E}_i:\frac{|A_i|}{2}-\frac{1}{40r^2}\sqrt{|A_i|}\leq |A_i\cap S|\leq \frac{|A_i|}{2}+\frac{1}{40r^2}\sqrt{|A_i|}\Big]\geq \frac{1}{40r^2},
\]
and
$$
\mathbb{P}\Big[\mathbf{E}_{s+1}:\frac{|B|}{2}+\frac{1}{20r}\sqrt{|B|}\leq |B\cap S|\leq \frac{|B|}{2}+\frac{1}{10r}\sqrt{|B|}\Big]\ge \frac{1}{30r}.
$$
Thus, conditioned on all events $\mathbf{E}_i$ for $i\in [s+1]$, by a direct calculation we have 
$$
%-3\alpha |S|\leq 
|A_i\cap S|-\frac{|S|}{r}<0\ \ \text{for each}\ i\in [s]$$
and
$$-3\alpha \frac{|S|}{r}\leq |B\cap S|-{(r-s)\frac{|S|}{r}}<3\alpha \frac{|S|}{r}. 
$$
It follows that~\ref{B1} holds and~\ref{B2} becomes trivial, while 
$${\mathbb{P}[\mathbf{E}]\geq \mathbb{P}\big[\mathbf{E}_0\mathbf{E}_1\mathbf{E}_2\cdots\mathbf{E}_{s+1} \big]\geq \mathbb{P}\Big[{\bigcap_{i\in [s+1]}\mathbf{E}_i }\Big]+
\mathbb{P}[\mathbf{E}_0 ]-1\ge  \frac{1}{30r(40r^2)^{r-1}}-\frac{1}{n}.}$$
%$$
%\frac{|S|}{r}\geq \frac{\sum_{i\in[r]}|A_i\cap S|+|B\cap S|}{r}\geq \frac{n}{2}-\frac{r-1}{40r^3}\sqrt{n}+\frac{1}{20r^2}\sqrt{ n}>|A_i\cap S|\ \text{for each}\ i\in [s].
%$$
%Moreover,
%$$
%|B\cap S|-\frac{|S|}{r}\leq \frac{|B|}{2}+\frac{\sqrt{n}}{10r}-\Big(\frac{n}{2}+\frac{1}{40r^2}\sqrt{n}\Big)\leq \frac{\sqrt{n}}{10r}.
%$$

%Together with Lemma \ref{lem:covering bad vertices for balance}, we obtain that $G[S]$ has a $K_r$-factor with probability at least $\frac{1}{(40r^2)^r}$.}}

\medskip
Combining Cases~1-3, we derive $\mathbb{P}[\mathbf{E}]\geq \frac{1}{35r(40r^2)^{r-1}}$.  Notice that $\mathbb{P}[|S|\equiv0~(\text{mod}~r)]=\frac{1}{r}$. Together with Lemma~\ref{lem:covering bad vertices for balance}, we obtain that $G[S]$ has a $K_r$-factor with probability at least $\frac{1}{(40r^2)^r}$. 
\end{proof}

\section{Proof of Lemmas}\label{sec:lemma}
In this section, we give the proofs of Lemma~\ref{lem:construct good-partition}, Lemma~\ref{lem:covering bad vertices for balance} and Lemma~\ref{vertex-cover}, respectively. 
\subsection{Proof of Lemma~\ref{lem:construct good-partition}}

\begin{proof}[{\bf Proof of Lemma~\ref{lem:construct good-partition}}]
Choose constants $r\geq 3,n\in \mathbb{N}$ and $\gamma_1,\gamma_2,\dots,\gamma_r$ such that 
\begin{align*}%\label{eq:p2}
    \frac{1}{n}\ll \gamma \ll \gamma_1\ll \dots \ll \gamma_r\ll\frac{1}{r}.
\end{align*}
Suppose that $G$ is an $rn$-vertex graph with $\delta(G)\geq (r-1)n+1$ and $G$ contains a $\gamma$-independent $n$-set. We proceed to greedily choose a maximal family of vertex-disjoint $n$-sets $A_1^0, \dotsc, A_s^0$ such that each $A_i^0$ is $\gamma_i$-independent.
Let $B^0 := V(G) \setminus \bigcup_{i \in [s]} A_i^0$ and denote $\mathcal{P}_0:=\{A_1^0,\ldots,A_s^0,B^0\}$.
Note that $G[B^0]$ has no $\gamma_{s+1}$-independent $n$-set and $|B^0|=(r-s)n$.

We additionally choose 
\begin{align*}%\label{eq:p3}
    \gamma_s\ll \alpha \ll\beta\ll \gamma_{s+1}.
\end{align*}
It is easy to see that for distinct $i,j\in [s]$, any vertex $v\in V(G)\setminus (A_i^0\cup A_j^0)$ cannot be both $(2\beta,A_i^0)$-bad and $(2\beta,A_j^0)$-bad. Next, we estimate the number of bad or exceptional vertices with respect to the partition $\mathcal{P}_0$. % and there is no~\ref{bad vtx3} and~\ref{bad vtx4} if $s=r$ in the following claim.
\begin{claim}\label{fact:bad vertices for Ai}
For each $i\in [s]$, the following hold:
\begin{enumerate}[label=(\arabic*)]
    \item\label{large-degree vtx} the number of $(\alpha^{1/4},A_i^0)$-good vertices in $A_i^0$ is at least $|A_i^0|-\alpha n$;
    \item\label{bad vtx2} the number of {$(2\beta,A_i^0)$}-bad vertices in $V(G)\setminus A_i^0$ is at most ${\alpha}n$;
    \item\label{bad vtx3} the number of {$(2\beta,B^0)$}-bad vertices in $V(G)\setminus B^0$ is at most $(r-1){\alpha} n$ if $s=r-1$, and there is no $({2\beta},B^0)$-bad vertex in $V(G)\setminus B^0$ if $s<r-1$;
    %\item\label{bad vtx3} any vertex $v\in V(G)$ cannot be both $(\beta,A_i)$-bad and $(\beta,A_j)$-bad for distinct $i,j\in [s]$;
    \item\label{bad vtx4} the number of $(\alpha^{1/4},B^0)$-exceptional vertices in $B^0$ is at most $\alpha n$.
\end{enumerate}
\end{claim}

\begin{proof}
(1) Recall that $A_i^0$ is a $\gamma_i$-independent set of size $n$ for each $i\in [s]$. As $\gamma_i\ll \alpha$, we obtain that the number of vertices in $A_i^0$ that are not $(\alpha^{1/4},A_i^0)$-good is at most 
$
\frac{2\gamma_i n^2}{\alpha^{1/4}n}\leq \alpha n$,
as desired.\medskip

(2) Since {$\delta(G)\geq (r-1)n+1$} and $A_i^0$ is a $\gamma_i$-independent set, the number of edges in $G[A_i^0,V(G)\setminus A_i^0]$ is at least
$$\underset{v\in A_i^0}{\sum}d(v)-2e(G[A_i^0])\geq 
|A_i^0| \cdot ((r-1)n+1)-2 e(G[A_i^0]) \ge |A_i^0||V(G)\setminus A_i^0|-2\gamma_i n^2.
$$
This implies that there are at most $2\gamma_i n^2$ non-edges between $A_i^0$ and $V(G)\setminus A_i^0$.
Note that every $(2\beta,A_i^0)$-bad vertex in $V(G)\setminus A_i^0$ contributes at least $(1-2\beta) n$ non-edges between $A_i^0$ and $V(G)\setminus A_i^0$.
As $\gamma_i \ll \alpha \ll \beta$,  we have that the number of $(2\beta,A_i^0)$-bad vertices in $V(G)\setminus A_i^0$ is at most 
$\frac{2\gamma_i n^2}{(1-2\beta) n}\le{\alpha}n$. \medskip

(3) If $s< r-1$, then $d(v,B^0)\ge |B^0|-n+1\ge n+1$ for every $v\in V(G)$. Hence, $G$ contains no $(2\beta,B^0)$-bad vertex. If $s=r-1$ and $v\in A_i^0$ is $(2\beta,B^0)$-bad, then $v$ cannot be an $(\alpha^{1/4},A_i^0)$-good vertex. Together with (1), we conclude that the number of $(2\beta,B^0)$-bad vertices is at most $(r-1){\alpha} n.$\medskip
% $$|V_{ex}(\mathcal{P}_0,\beta,i)|\leq |V_{ne}(\mathcal{P}_0,\alpha^{1/5},i)|\le \frac{2\gamma_i n^2}{\alpha^{1/5}n}\leq  \alpha n.$$

%(3) Let $u\in A_i^0$ be a $(\beta,A_j)$-bad vertex for distinct $i,j\in [s]$. Then $d(u,A_{\ell}^0)\geq n-\beta n$ for all $\ell\in [s]\setminus \{j\}$ due to the regularity of $G$. 

(4) For each $i\in [s]$, let $B_i\<B^0$ be the set of all $(\alpha^{1/4},B^0)$-exceptional vertices $v$ such that $d(v,A_i^0)\le |A_i^0|-\alpha^{1/4}n$ and $B_e=\bigcup_{i\in [s]}B_i$. Now we consider the size of $B_i$.
Note that $|B^0|=(r-s)n$.
Thus 
\[
e(G[A_i^0,B^0])\le |B_i|\cdot \big(|A_i^0|-{\alpha^{1/4} n}\big)+\big((r-s)n-|B_i|\big)\cdot|A_i^0|=|A_i^0|(r-s)n-|B_i|\alpha^{1/4}n.
\]
Recall that $\delta(G)\geq (r-1)n+1$ and $A_i^0$ is a $\gamma_i$-independent set. Then % Combing with $d(v)=(r-1)n+1$ for each $v\in A_i$, we obtain that
\[
2\gamma_in^2\geq 2e(G[A_i^0])\ge|A_i^0|\cdot\big((r-1)n+1-(s-1)n\big)-e(G[A_i^0,B^0])\ge |A_i^0|+|B_i|\alpha^{1/4}n=n+|B_i|\alpha^{1/4}n,
\]
which implies $|B_i|\leq \frac{\alpha}{r} n$ as $\gamma_i\ll \alpha\ll \frac{1}{r}$. Thus 
%Recall that $A_i^0$ is a $\gamma_i$-independent set of $G$ and at most $r\alpha n$ vertices are moved from other part to $A_i^0$.
%As $\gamma_i\ll\alpha\ll\frac{1}{r}$, we have
%$$
%e(G[A_i])\leq (\gamma_i+2r\alpha) n^2<\sqrt{\alpha}n^2.
%$$
%Thus $|A_i|+|B_i|\cdot\beta n\le e(G[A_i])\le \sqrt{\alpha}n^2$, which implies that
$
|B_e|\leq \sum_{i\in [s]}|B_i|\le \alpha n.$ 
%as $\alpha\ll\beta$.
\end{proof}

Now, we move all $(2\beta, D^0)$-bad vertices from $V(G)\setminus D^0$ to $D^0$ for each $D^0\in \mathcal{P}_0$, and denote the resulting partition by $\mathcal{P}_1:=\{A_1^1,\ldots,A_s^1,B^1\}$. {As $\alpha\ll\beta$, Claim~\ref{fact:bad vertices for Ai} (2) implies} that every vertex in $D^1\in \mathcal{P}_1$ has at least {$\frac{3}{2}\beta n$} neighbors in each $D'$ with $D'\in \mathcal{P}_1\setminus \{D^1\}$. By Claim~\ref{fact:bad vertices for Ai} (2)-(3), one has $|A_i^1|\leq n+\alpha n$ for each $i\in [s]$ and $|B^1|\leq (r-s)n+(r-1)\alpha n$. %Based on Claim \ref{fact:bad vertices for Ai}, we obtain that \ref{B1}, \ref{B04} and \ref{B05} hold. 
%Next, we will proceed with the following two-step operations of moving at most $\alpha n$ vertices. The first step aims to obtain a new partition satisfying the degree condition in Lemma~\ref{lem:balance}. It is obvious that the new partition may be unbalanced. So the purpose of Step 2 is to identify a collection of matchings in the \textit{large} parts, whereby we adjust the partition to be balanced.

%\medskip
%{\bf Step 1. Move all $(\beta, D)$-bad vertices from $V(G)\setminus D$ to $D$ for each $D\in \mathcal{P}_0$.}
%\medskip

%For each $D\in \mathcal{P}$, if $v\in V(G)\setminus D$ is $(\beta,D)$-bad, then we move $v$ to $D$. 
%In the process of moving vertices, there is a tiny difference between the cases $s< r-1$ and the others in whether to move vertices to $B^0$. To be precise, i
%Notice that if $s< r-1$, then $d(v,B^0)\ge |B^0|-n+1\ge n+1$ for every $v\in V(G)$. In this case, we do not need to move any vertices to $B^0$. If $s\ge r-1$, then there may exists $v\in A_i$ such that $d(v,B^0)\le \beta n$ for some $i\in [s]$ (such a vertex cannot be $(\beta,A_i)$-good). In this case, we move such vertex $v$ into $B^0$. Note that the resulting partition after the above moving procedure may not be balanced and we denote it as $\mathcal{P}_1:=\{A_1^1,\ldots,A_s^1,B^1\}$, in which every vertex in $A_i^1$ has large degree to the outside of $A_i^1$. 
Recall that for each $i \in [s]$, a part $A_i^1$ is called {large} if $|A_i^1| > n$, and {small} otherwise. 
Next, we will move vertices in large parts to satisfy~\ref{B2}.  
%\medskip
%{\bf Step 2. Move not $(\beta, A_i)$-bad vertices from large part $A_i^1$ to small one $A_j^1$.}
%\medskip
Let $A_i^1$ be a large part. % set with size greater than $n$. 
If $\Delta(G[A_i^1])\leq \beta n$, then we are done. Otherwise, move a vertex in \(A_i^1\) which has at least \( \beta n \) neighbors in \(A_i^1\) to any part \(A_j^1\) with $|A_j^1|<n$, or $B^1$ if $|B^1|<(r-s)n$, and update \(\mathcal{P}^1\) accordingly. %Assume that $|A_i^1|>n$ and $\Delta(G[A_i^1])\leq \beta n$ for some $i\in [s]$.  arbitrarily move each vertex in $A_i^1$ which has at least $\beta n$ neighbors in $A_i^1$ to other side $A_j^1$, where $j\in [s]$ with $|A_j^1|<n$, and update $\mathcal{P}^1$ accordingly. 
Repeat this until we reach the point {where} either $|A_j^1|=n$ for each $j\in [s]$ and $|B^1|=(r-s)n$, or $G[A_i^1]$ has maximum degree at most $\beta n$ whenever $|A_i^1|>n$ for $i\in [s]$; denote the resulting partition by $\mathcal{P}:=\{A_1,\ldots,A_s,B\}$. 
Obviously,~\ref{B1} holds. Moreover,  each vertex has  at least $(r-s-1)n-r\alpha n$ neighbors in $B$ and every vertex in $D$ has at least $\beta n$ neighbors in each $D'$ with $D'\in \mathcal{P}\setminus \{D\}$, that is,~\ref{B04}-\ref{B05} hold.

Recall that $\de(G)\geq (r-1)n+1$. Then $\de(G[A_i])\geq |A_i|-n+1$ when $|A_i|\geq n$ with $i\in [s]$.  %Let $|A_i|=n+k_i$ for each $i\in [s]$. 
If $|A_i|>n$, then together with $\Delta(G[A_i])\leq \beta n$, Remark~\ref{remark:vertex-cover} implies that $G[A_i]$ has a matching of size at least
$$
\frac{\de(G[A_i])|A_i|}{4\Delta(G[A_i])}\geq \frac{(|A_i|-n+1)|A_i|}{4\beta n}\geq \frac{|A_i|-n+1}{5\beta}\ge |A_i|-n+r.
$$
That is,~\ref{B2} holds. 

{By Claim~\ref{fact:bad vertices for Ai} (1)-(2) we obtain~\ref{B03}, and by (3)-(4) we obtain~\ref{B06}. Recall that %each $A_i^0$ is a $\gamma_i$-independent set, 
$B^0$ contains no $\gamma_{s+1}$-independent $n$-set and $\alpha\ll\gamma_{s+1}$. %Then %for each $i\in [s]$ we have 
%$$
%e(G[A_i])\leq e(G[A_i^0])+|A_i\setminus A_i^0||A_i|\leq \gamma_in^2+2\alpha n\cdot (n+\alpha n)\leq 3\alpha n^2. 
%$$
%Moreover, 
Let $B'$ be a subset of $B$ with size $\lceil\frac{|B|}{r-s}\rceil$. If $|B'\cap B^0|\geq n$, then $e(G[B'])\geq e(G[B'\cap B^0])\geq \gamma_{s+1}n^2$.  If $|B'\cap B^0|< n$, then there is a subset $B''$ such that  $B'\cap B^0\subseteq B''\subseteq B^0$ with $|B''|=n$. Therefore,  
$$
e(G[B'])\geq e(G[B''])-|B''\setminus B'||B''|\geq \gamma_{s+1}n^2-r\alpha n\cdot n\geq \frac{\gamma_{s+1}}{2}n^2.
$$
That is, $G[B]$ contains no  $\frac{\gamma_{s+1}}{2}$-independent set of size at least $\frac{|B|}{r-s}$.  Hence,~\ref{B07} holds by taking $\gamma':=\frac{\gamma_{s+1}}{2}$. Thus, $\mathcal{P}$ is an  $(\alpha,\beta,\beta,\gamma')$-good partition of $G$, as desired.} 
%Since Together with Claim \ref{fact:bad vertices for Ai}, it is straightforward to check that~\ref{B1}-\ref{B06} hold and~\ref{B07} holds with parameters $2\gamma_s$ and $\frac{\gamma_{s+1}}{2}$.  %(you chose $\gamma_{s+1}\ll \gamma'$ at the beginning!)}. 
%\medskip
%After the two-step moving operation, we state the following fact that a large matching can be found in a large part.
\end{proof}
%Let $r,n,s$ be positive integers with $r\ge 3,\ 1\le s\le r$ and let $\gamma$ be a positive constant. Define constants $\gamma, \gamma_1,\gamma_2,\dots,\gamma_r, \gamma_{r+1}$ satisfying 
%\begin{align}\label{eq:p2}
%    0<\frac{1}{n}\ll \gamma \ll \gamma_1\ll \dots \ll \gamma_r\ll\gamma_{r+1}=\frac{1}{r}.
%\end{align}

%Since $G$ contains a $\gamma$-independent $n$-set, we may greedily choose vertex-disjoint $\gamma_i$-independent $n$-sets $A_1^0, \dotsc, A_s^0$, until no $\gamma_{s+1}$-independent $n$-set remains. 
%Let $B^0 := V(G) \setminus \bigcup_{i \in [s]} A_i^0$ and we obtain a partition $\mathcal{P}_0:=\{A_1^0,\ldots,A_s^0,B^0\}$ of $V(G)$.
%Note that $B^0$ has no $\gamma_{s+1}$-independent $n$-set and $|B^0|=(r-s)n$.

%When the constant $s$ is fixed, we define 
%\begin{align}\label{eq:p3}
%    \gamma_s\ll \alpha \ll\beta\ll \gamma_{s+1}.
%\end{align}

%In the rest of the proof, our main challenge is to adjust the partition $\mathcal{P}$ slightly to satisfy the following degree condition.

\subsection{Proof of Lemma~\ref{lem:covering bad vertices for balance}}

Keevash and Mycroft \cite{KM-multi} proved the following multipartite version of Hajnal-Szemer\'edi theorem, which provides the best possible degree condition that guarantees a transversal $K_r$-factor in a balanced $r$-partite graph.%The following multi-partite version of Hajnal-Szemer\'edi theorem provides a degree condition that guarantees the existence of $K_r$-factors in a balanced $r$-partite graph. %, which follows from \cite[Lemma 6.3]{gan2024} by using Edmonds algorithm.{(why don't we use the exact degree condition of Keevash and Mycroft?)}

\begin{lemma}[\cite{KM-multi}]\label{lem:HS thm in multipartite graph} 
Given $r \in \mathbb{N}$ there exists an $n_0 \in \mathbb{N}$ such that the following holds. Suppose $G$ is an $r$-partite graph with vertex classes $V_1, \dots, V_r$ where $|V_i| = n \geq n_0$ for all $1 \leq i \leq r$. If
\[
\min_{i\in [r]}\min_{v\in V_i}\min_{j\in [r]\setminus \{i\}} |N(v)\cap V_j| \geq \Big(1 - \frac{1}{r}\Big)n + 1,
\]
then $G$ contains a $K_r$-factor.
\end{lemma}
We will use this to prove that a $K_r$-factor can also be found in an almost balanced partition,  provided that every large part {contains a matching with enough edges.} In fact, this enables us to first pick a few vertex-disjoint copies of $K_r$ each containing exactly one `extra' vertex in a large part (compared to the balanced copies) so as to obtain a balanced $r$-partite subgraph, to which Lemma~\ref{lem:HS thm in multipartite graph} would be applied. 

For any \( i \in [r] \), let \( \mathbf{u}_i\in \mathbb{R}^r \) be the \( i \)-th unit vector, i.e.\ \( \mathbf{u}_i \) has $1$ on the \( i \)-th coordinate and $0$ on the other coordinates.  Let \( \mathbf{1}_r\in \mathbb{R}^r \) be an $r$-vector {with all entries equal to $1$.}
% all of whose entries are $1$. 
Given a vertex partition $\mathcal{P}=\{V_1,\ldots,V_r\}$ of $G$ and $F\<G$,  the \textit{index vector} of \( F \)  (with respect to $\mathcal{P}$) is \( \mathbf{i}(F) = (x_1, x_2, \ldots, x_{r}) \) where \( x_i = |V(F) \cap V_i| \) for each \( i \in [r] \). %In the following, we give a sufficient condition that guarantees the existence of $K_r$-factors in $G$.
%Applying with Lemma~\ref{lem:HS thm in multipartite graph}, we have the following result to guarantee the existence of $K_r$-factors in $G$. 

%Without loss of generality, assume that $|A_i|=n+k_i$ for each $i\in [s]$ and $k_1\leq \cdots\leq k_s$. Note that $\sum_{i\in [s]}k_i=0$. We have the following result for unbalanced partition.
\begin{lemma}\label{lem:balance}
Let $\frac{1}{n}\ll \alpha \ll \beta \ll\frac{1}{r}$. Suppose that $G$ is a graph on $rn$ vertices with a partition $V(G) =\{A_1, A_2,\ldots, A_r\}$ satisfying 
\begin{enumerate}
[label =\rm  (B\arabic{enumi})]
    \item\label{A1} $ |A_1|\leq \cdots\leq |A_r|\leq n+\alpha n$;
    
    \item\label{A3} for any distinct $i,j \in [r]$ and any $v \in A_i$, we have $d(v, A_j) \ge (1-\beta)|A_j|$;
    \item\label{A2} for each $i \in [r]$, $G[A_i]$ contains a matching $M_i$ of size $k_i:=|A_i|-n$.
\end{enumerate}
{Then $G$ contains a $K_r$-factor $\mathcal{F}$ such that $\bigcup_{i\in [r]}E(M_i)\subseteq E(\mathcal{F})$,  and each $K_r$ in $\mathcal{F}$ satisfies    
\begin{itemize}
    \item either $|V(K_r)\cap A_i|=1$ for all $i\in [r]$;
    \item  or $|E(K_r)\cap E(M_i)|=1$ for exactly one $i\in [r]$ and $|V(K_r)\cap A_j|=1$ for every $j\in [r]\setminus \{i\}$ with $|A_j|\geq n$.
\end{itemize}}
 %{Moreover, each edge in the  matching of $G[A_i]$ lies in distinct copies of $K_r$ }
\end{lemma}

\begin{proof}%[Proof of Lemma~\ref{lem:balance}]
Suppose that $|A_j|=|A_1|+a_j$ for each $j \in [2, r]$ and $|A_1|=a_1$.
Without loss of generality, {we may assume $A_1,A_2,\ldots,A_p$ are all the parts of size at most $n$ for  some $p\in[r]$}.
If $a_j=0$ for each $j \in [2, r]$, that is, $\{A_1,A_2, \ldots, A_r\}$ is a balanced partition, then Lemma~\ref{lem:HS thm in multipartite graph} implies that $G$ contains a desired $K_r$-factor.
Otherwise, recall that for each $i \in [p+1,r]$, $G[A_i]$ contains a matching $M_i$ of size $k_i:=|A_i|-n$, we proceed in two steps. 
\begin{itemize}
    \item[$(1)$]  For each $i \in [2, p]$, iteratively select $n - |A_i|$ vertex-disjoint copies of $K_r$ in $G$, {each of which} has an index vector $\mathbf{1}_r - \mathbf{u}_i + \mathbf{u}_\ell$ for some $\ell\in[p+1, r]$,  
provided that $M_{\ell}$ has an edge not appearing in previously selected copies of $K_r$.

    \item[$(2)$]  For each $i \in [p+1, r]$, pick $k_i - k_i'$ vertex-disjoint copies of $K_r$ with index vector $\mathbf{1}_r - \mathbf{u}_1 + \mathbf{u}_i$, where $k_i'$ is the total number of edges of $M_i$ used in step~$(1)$.
\end{itemize}

{Indeed, in step~$(1)$, it necessarily holds that $$\sum_{i\in [2,p]}(n-|A_i|)\leq \sum_{i\in [p]}(n-|A_i|)= \sum_{i\in [p+1,r]}(|A_i|-n)=\sum_{i\in [p+1,r]}k_i.$$
Meanwhile, the degree condition in~\ref{A3} allows a greedy vertex-selection of each $K_r$ starting with an edge (in both steps). The resulting $K_r$-tiling from above is denoted as $\mathcal{K}$. Note that  every edge in $M_i$ lies in exactly one  clique of $\mathcal{K}$, and such a clique  contains exactly  one vertex from each $A_j$ with $j\in [r]\setminus \{i\}$ and $|A_j|\geq n$.}  
% {(explain in details! Are there enough matching edges?)}
We now show that the resulting partition $\mathcal{P}'=\{A_1', \ldots, A_r'\}$ is balanced, where $A_i':=A_i\setminus V(\mathcal{K}).$
\begin{claim}\label{result-balance}
    For each $\ell\in [r]$, we have $|A_{\ell}'|=a_1 + \sum_{i\in [2,p]} k_i$.
\end{claim}
\begin{proof}
Notice that $\sum_{i\in [r]}k_i=0$ and $\sum_{i\in [p+1,r]}k_i'=\sum_{i\in [2,p]}(n-|A_i|)$. Hence %$-\sum_{i\in [p]}k_i=\sum_{i\in [p+1,r]}k_i$. We verify the number of remaining vertices in each part by   direct calculation. 
\begin{itemize}
    \item $|A_1'|= a_1-\sum_{i=2}^p(n-|A_i|)=a_1+\sum_{i=2}^pk_i$;
    \item for each $\ell\in [2,p]$, 
        \begin{align*}
        |A_{\ell}'|=&a_1+a_\ell-\sum_{i\in [2,p]\setminus\{\ell\}}(n-|A_i|)-\sum_{i\in [p+1,r]}(k_{i}-k_i')\\
        =&n-\sum_{i\in [p+1,r]}k_{i}=a_1-k_1+\sum_{i\in [p]}k_{i}=a_1+\sum_{i\in [2,p]}k_i;
        \end{align*}
        
    \item for each $\ell\in [p+1,r]$,                  \begin{align*}
        |A_{\ell}'|=&a_1+a_\ell-\sum_{i\in [2,p]}(n-|A_i|)-k_{\ell}'-\sum_{i\in [p+1,r]}(k_{i}-k_i')-(k_{\ell}-k_{\ell}')\\
        =&|A_{\ell}|-\sum_{i\in [p+1,r]}k_{i}-k_{\ell}=n-\sum_{i\in [p+1,r]}k_{i}=a_1+\sum_{i\in [2,p]}k_i,
        \end{align*}
\end{itemize}
as desired. 
\end{proof}

By~\ref{A1}, for each $i\in [r]$ one has 
$$|V(\mathcal{K})\cap A_i|\leq 2\sum_{i\in [2,p]}(n-|A_i|)+2\sum_{i\in [p+1,r]}(k_i-k_i')\leq 4\sum_{i\in [p+1,r]}k_i\leq 4r\alpha n.$$ % we remove at most $4\alpha n{??}$ vertices in each part $A_i$.
Hence $|A_i'|\geq |A_i|-4r\alpha n$. Together with~\ref{A3} and the choice  $\alpha\ll\beta\ll\frac{1}{r}$, for any distinct $i,j \in [r]$ and any $v \in A_i'$, we have
\[
d(v, A_i') \ge |A_i'|-{2\beta}n\ge \Big(1-\frac{1}{r}\Big)|A_i'|+1.
\]
Applying Theorem~\ref{lem:HS thm in multipartite graph} to $G - V(\mathcal{K})$ yields a $K_r$-factor {such that each copy of $K_r$ contains exactly one vertex from each $A_i$ with $i\in [r]$}. This together with $\mathcal{K}$ gives a desired  $K_r$-factor of $G$. 
\end{proof}
%We call a part $A_i\in \mathcal{P}$ a \textit{large} part if $|A_i|> \frac{|G|}{r}$, and a \textit{small} part otherwise. 
%In order to use Lemma~\ref{lem:balance} to find a \(K_r\)-factor in \(G\), we need to construct a graph satisfying all of \ref{A1}-\ref{A2}.  
To apply Lemma~\ref{lem:balance} in constructing a \(K_r\)-factor of \(G\), 
we shall construct a graph satisfying conditions~\ref{A1}-\ref{A2}. Let $\mathcal{P}=\{A_1,\ldots,A_s,B\}$ be a good partition of $G$. %{\sout{We call a copy of $K_r$ \textit{balanced} if it has index vector $(1,\cdots,1,r-s)$}}{(Yang:move it to where it is used)}. 
As outlined in Section~2, we proceed as follows. 

\begin{itemize}
    \item For~\ref{A3}, we construct a $K_r$-tiling $\mathcal{K}$ that covers all bad or exceptional vertices of $G$.
    \item For~\ref{A1}, we find a $\{K_{r-s},K_{r-s+1}\}$-factor in $G[B\setminus V(\mathcal{K})]$ and contract each copy to either an edge or a single vertex.
    \item For~\ref{A2}, the matching number is essentially derived from~\ref{B2} and the contraction as above.
\end{itemize}

{For $r-s\geq 3$, Theorem~\ref{thm:Non-extremal case r>3} gives a $K_{r-s}$-factor in $G[B']$, where $B'$ denotes the largest subset of $B\setminus V(\mathcal{K})$ with $(r-s)\mid |B'|$.  %desired  When $r-s\ge 3$, by~\ref{B05} and~\ref{B07}, we obtain that Theorem~\ref{thm:Non-extremal case r>3}  implies that there is a $K_{r-s}$-factor in $G[B]-V(\mathcal{K})$.
For $r-s=2$, by applying the subsequent  result of~\cite{gan2024} to $G[B']$, we obtain that either $G[B']$ contains a perfect matching, or it is disconnected with  two odd components. 
%In the latter case, by applying some minor modification, % we construct two copies of $K_r$ with index vectors $(1,\dots,1,2,1)$ and $(1,\dots,1,0,3)$, or do some minor modification. Applying Lemma~\ref{thm:r-s=2} again 
In the latter case, after some minor adjustments, we can ensure that the remaining graph contains a perfect matching.} 
%Indeed, on the one hand,~\ref{B2} or Lemma~\ref{vertex-cover} gives us that there is an edge in some $A_i$ and we extend it to $H_1$ by using one vertex from one odd component in $G[B]$; on the other hand, we can find a copy of $K_3$ in the other odd component in $G[B]$ and we extend it to $H_2$.}
% Let $\epsilon > 0$. Suppose that $G$ is a graph on $2n$ vertices. We say that $G$ is $\epsilon$-\textit{close} to $2K_n$ if there exists a partition $A, B$ of $V(G)$ such that $|A| = n$, $|B| = n$ and $|E(G[A,B])|\le \epsilon n^2$.

\begin{lemma}[\cite{gan2024}]\label{thm:r-s=2}
Let $\beta>0$ and $n\in \mathbb{N}$ such that $\frac{1}{n}\ll\beta$. Suppose that $G$ is a graph on $2n$ vertices with $$\de(G)\ge (1-\beta)n.$$ Then at least one of the following holds:
\begin{itemize}
    \item $G$ contains a $2\beta$-independent set of size at least $n$;
    \item {$G$ is disconnected and consists of two odd components;}
    % $G$ is $3\beta$-close to $2K_n$, {where $n$ is odd};
    \item $G$ contains a perfect matching.
\end{itemize}
\end{lemma}
The following lemma allows us to find many vertex-disjoint copies of $K_r$ under a slightly weaker minimum degree condition. 
\begin{lemma}[\cite{Treg2015}]\label{SuperT}
    Let $\frac{1}{n} \ll \eps \ll \alpha \ll \frac{1}{r}$ where $r\in \mathbb{N}$ and $r\ge 2$. Suppose that $G$ is an $n$-vertex graph that contains at most $\eps n^r$ copies of $K_r$ and 
$$\delta(G)\ge \Big(1-\frac{1}{r-1}-\alpha\Big)n.$$
%and so that $G$ contains at most $\eps n^r$ copies of $K_r$. 
Then $G$ admits a $\sqrt{\alpha}$-independent set of size at least $\frac{n}{r-1}$.
\end{lemma}

%In the following sections, we are ready to construct a partition of independent sets firstly (in section 4.1) and then transform it into a balanced one through a series of appropriate operations (in section 4.2).

In the following, we will use Lemma~\ref{lem:balance} to prove Lemma~\ref{lem:covering bad vertices for balance}. 

\begin{proof}[{\bf Proof of Lemma~\ref{lem:covering bad vertices for balance}}]
Choose $r\geq 3,n\in \mathbb{N}$ and constants satisfying %{replace $\gamma_1,\gamma_2$ with $\gamma,\gamma'$? } 
$$\frac{1}{n}\ll \alpha\ll\eps\ll\beta',  \beta\ll\gamma\ll \frac{1}{r}.$$ 
Assume that $G$ is an $rn$-vertex graph with $\de(G)\geq (r-1)n-\alpha n$ and $G$ admits an $(\alpha,\beta,\beta',\gamma)$-good partition $\mathcal{P}=\{A_1,\ldots,A_s,B\}$ {for some  $s\in [r]$}. Without loss of generality, assume that $|A_1|\leq |A_2|\leq \cdots\leq |A_s|$  and assume $A_1,\ldots,A_p$ are all the parts of size at most $n$ for some $p\in[s]$. 
For each $i \in [s]$, if $|A_i| > n$, then let $M_i$ be a matching of size $|A_i| - n + r$ in $G[A_i]$ (its existence follows from~\ref{B2}); {if $|A_s| = n$, then let $M_s$ be an edge in $G[A_s]$ containing at most one vertex that is not $(\alpha^{1/5},A_s)$-good (this also follows from~\ref{B2} together with the definition of $(\alpha^{1/5},A_s)$-good vertex and~\ref{B03});} %\sout{and $G[A_s]$ contains an edge whose both endpoints are $(\alpha^{1/5}, A_s)$-good}%\todo{This is not stated in $B_2$}, then let $M_s$ to be that edge;}
otherwise we let $M_i = \emptyset$. Define $U = \bigcup_{i \in [s]} V(M_i)$ as the union of the vertex sets of these matchings. By~\ref{B1}, one has $|U| \leq 3r\alpha n$. %Choose additional 
Our goal is to construct an auxiliary graph and a corresponding partition as required in Lemma~\ref{lem:balance}. 

\medskip
{\bf Step 1.~Greedily cover all bad or exceptional vertices using vertex-disjoint copies of $K_r$ whilst avoiding any vertex in $U$.}
\medskip%

Notice that $\Delta(G[A_i])\leq \beta' n$ if $|A_i|>n$. This together with the minimum degree condition gives that $d(v,D)\geq |D|-2{\beta'} n$ for any $v\in A_i$ with $|A_i|> n$ and any $D\in \mathcal{P}\setminus \{A_i\}$, {which satisfies~\ref{A3}}. It remains to consider the following two types of vertices:  %{in $B$ and $A_i$ with $|A_i|\le n$}:
\begin{enumerate}[label =\rm  {\bf Type \arabic{enumi}.}, ref = \rm {\bf Type \arabic{enumi}},  leftmargin=6em]
    \item\label{v1} $v\in A_i\setminus U$ and $v$ is not $(\alpha^{1/5},A_i)$-good  for some $i\in [s]$ with $|A_i|\leq n$;
    \item\label{v2} %$v\in B$ and $d(v,A_i)\leq |A_i|-\alpha^{1/5} n$ for some $i\in [s]$, namely, 
    $v\in B$ and $v$ is $(\alpha^{1/5},B)$-exceptional. 
\end{enumerate}

%Suppose for contradiction that there exists $v_i\in A_i\setminus U$ be of \ref{v1} that is not covered by $\mathcal{Q}$.
The following   claim guarantees the existence of a copy of $K_r$ that covers any fixed vertex while avoiding a prescribed vertex set.   
\begin{claim}\label{covering-v}
   Let $W\subseteq V(G)$ be a set of size at most $\frac{\beta}{3}n$ and $v_i$ be a vertex in $A_i\setminus W$ for some $i\in [s]$. Then there is a copy of $K_r$ in $G-W$ with index vector $(1,\ldots,1,r-s)$ containing $v_i$.  
\end{claim}
\begin{proof}
Let $v_i\in A_i\setminus W$ for some $i\in [s]$ {and choose \( j \in [s] \setminus \{i\} \). It follows from~\ref{B04} that} \( d(v_i, A_j) \geq \beta n \), and from \ref{B03} that at most \( 2\alpha n \) vertices in \( A_j \) are not \( (\alpha^{1/5}, A_j) \)-good. As $\alpha\ll\beta$, there are at least \( (\beta - 2\alpha -\frac{\beta}{3})n\ge \frac{\beta}{2}n \) vertices in \( N(v_i, A_j)\setminus W \) that are \( (\alpha^{1/5}, A_j) \)-good; we  fix one such vertex as \( v_j \). {Given} $\ell\in [s] \setminus \{i, j\}$, the minimum degree condition \( \delta(G) \geq (r-1)n - \alpha n \) implies \( d(v_j, A_\ell) \geq |A_\ell| - 2\alpha^{1/5} n \).
%{We pick exactly one \( (\alpha^{1/5}, A_\ell) \)-good vertex $v_{\ell}$ greedily for each \( \ell \in [s] \setminus \{i, j\} \) in the common neighbor of chosen vertices in $A_{\ell}$, while avoiding vertices in $W$. Since the number of previously chosen vertices is at most $s$, we have at least \( (\beta - 2s\alpha^{1/5} - 2\alpha-\frac{\beta}{3})n\ge \frac{\beta}{2}n \) choices for such good vertex in each step, which is adjacent to all previously fixed vertices.}
{Thus} at least \( (\beta - 2\alpha^{1/5} - 2\alpha-\frac{\beta}{3})n\ge \frac{\beta}{2}n \) vertices in \( N(\{v_i, v_j\}, A_\ell\setminus W) \) are \( (\alpha^{1/5}, A_\ell) \)-good; we fix one such vertex as \( v_{\ell} \). 
Iterating this process for all \( \ell \in [s] \setminus \{i,j\} \), one can take an \((\alpha^{1/5}, A_\ell) \)-good vertex \( v_{\ell}\in A_\ell\setminus W \), which is adjacent to all previously fixed vertices. This follows {from the fact that \( \alpha \ll \beta \), which guarantees at least \( \frac{\beta}{2}n \) available choices  at each step.}
Hence we obtain a clique on vertices, say $\{v_1,v_2,\ldots,v_s\}$.

Next, we are to find a copy of $K_{r-s}$ in $N(\{v_1,\ldots,v_s\},B) \setminus W$, and thus form a copy of $K_r$ covering the vertex $v_i$. Let $\Tilde{B}=N(\{v_1,\cdots,v_s\},B)\setminus W$. Similarly, 
$
|\Tilde{B}|\geq \frac{\beta}{2}n
.$ 
If $r-s\le 1$, then we are done. In what follows, we assume $r-s\ge 2$. Recall that $(r-s-r\alpha)n\leq |B|\leq (r-s+r\alpha)n$. % {and $|\Tilde{B}|\ge (r-s-1)n-2(s-1)\alpha^{1/5}n$}. 
Then~\ref{B05} implies that 
\begin{align*}
\de(G[B])&\ge (r-s-1)n-r\alpha n=\Big(1-\frac{1}{r-s}\Big)\cdot\big((r-s)n+r\alpha n\big)-\Big(2-\frac{1}{r-s}\Big)\cdot r\alpha n\\
&\ge \Big(1-\frac{1}{r-s}-\alpha^{1/5}\Big)\cdot|B|.
\end{align*}
%\medskip
%{\bf Covering bad vertices in $A_i$ with $|A_i|\le n$.}
%\medskip
%We cover each worse vertex $v_i$ in $A_i$ with a copy of vertex-disjoint $K_r$ if $d_{A_i}(v_i)\ge\alpha^{1/5} n$.
%For each bad vertex $v_i\in A_i$, the condition~\ref{B04} and~\ref{B05} in Lemma~\ref{lem:partition after adjust vertices} imply that $v_i$ has at least $\alpha^{1/5} n$ neighbors in each of the other parts of the partition. By Fact~\ref{fact:bad vertices for Ai}~\ref{large-degree vtx}-\ref{bad vtx2}, one can choose an $(\alpha^{1/5}, j)$-bad vertex $v_j$ in $A_j$ for each $j\neq i$ such that $v_1,v_2,\cdots,v_s$ induces a $K_s$ as $\alpha\ll\alpha^{1/5}$. Note that such vertex $v_j$ is adjacent to all but at most $\alpha^{1/5} n$ vertices in $V(G)\setminus A_j$. We now try to find a copy of $K_{r-s}$ in the common neighborhood of $v_1,v_2,\cdots,v_s$ in $B$. If $s=r$, then we are done since we have identified such $r$ vertices. Let $G'=N_B(\{v_1,v_2,\cdots,v_s\})$. Since $G$ is $((r-1)n+1)$-regular and $\alpha\ll\alpha^{1/5}$, we have
%\[
%\de(G[B])\ge (r-1)n+1-sn\ge \left(1-\frac{1}{r-s}-\alpha^{1/5}\right)\cdot|B|.
%\]
Let $x:=|B\setminus \Tilde{B}|$. {Since $v_j$ is $(\alpha^{1/5},A_j)$-good for each $j\in [s]\setminus \{i\}$, one has  
\begin{align}\label{eq:x}
    x\leq \Big|\bigcup_{i\in [s]}(B\setminus N(v_i))\Big|+|W|\leq  n+2(s-1)\alpha^{1/5} n+\frac{\beta}{3} n.
\end{align}
It follows that 
$|\Tilde{B}|\geq (r-s-1)n-\frac{\beta}{2}n$ as $\alpha\ll\beta$.} % , then $d(v_i,B)\ge |B|-n+1$ for each $i\in [s]$. Let $x$ be the number of non-neighbors of $v_1,v_2,\cdots,v_s$. As $v_j$ is a $(\beta,A_j)$-bad vertex in $A_j$ for each $i\neq j$, we have
%\[
%|B|-x=|G'|\ge |B|-n+1-(s-1)\alpha^{1/5} n,
%\]
%which means that $x\le n+(s-1)\beta n$.
Thus, \allowdisplaybreaks
%as $\alpha\ll\alpha^{1/5}\ll\eps\ll 1$ we have
\begin{align*}
\de(G[\Tilde{B}])&\ge \de(G[B])-x\ge \Big(1-\frac{1}{r-s}-\alpha^{1/5}\Big)\cdot|B|-x\\
&=|B|-x-\frac{|B|}{r-s}-\alpha^{1/5}|B|\\
&\ge|B|-x-\frac{(1+\beta)(|B|-x)}{r-s-1}-(|B|-x)\beta \\
&\geq\Big(1-\frac{1}{r-s-1}-2\beta\Big)\cdot(|B|-x)\\
&=\Big(1-\frac{1}{r-s-1}-2\beta\Big)\cdot|\Tilde{B}|,
\end{align*}
where the last inequality follows as % $x\le n+{2}(s-1)\alpha^{1/5} n$ and therefore 
$\frac{|B|}{r-s}\le\frac{(1+{\beta})(|B|-x)}{r-s-1}$ by~\eqref{eq:x}.
Based on~\ref{B07} and $\beta\ll \gamma$, we conclude that $G[\Tilde{B}]$ has no $\sqrt{\beta}$-independent set of size at least $\frac{|\Tilde{B}|}{r-s-1}$.  Applying Theorem~\ref{SuperT} with $(G[\Tilde{B}],\beta)$ in place of $(G,\alpha)$ yields a copy of $K_{r-s}$ in $G[\Tilde{B}]$. {This yields a copy of \(K_r\) with index vector $(1,\ldots,1,r-s)$ covering \(v_i\) in \(G-W\), as desired.} % contradicting the maximality of \(\mathcal{Q}\). Therefore, \(\mathcal{Q}\) is a \(K_r\)-tiling in $G-U$ that covers all vertices of \ref{v1}, and each copy of \(K_r\) in \(\mathcal{Q}\) has index vector \((1,\dots,1,r-s)\).} 
\end{proof}

By~\ref{B03}, there are at most $2r\alpha n$ vertices of~\ref{v1}. 
Combining Claim~\ref{covering-v} with the fact that $|U|\leq 3r\alpha n$, {we conclude that} there exists a \(K_r\)-tiling \(\mathcal{Q}\) in \(G-U\) covering all vertices of~\ref{v1}, such that each copy of \(K_r\) has index vector \((1,\dots,1,r-s)\) and contains exactly one vertex of~\ref{v1}.  
Then $|V(\mathcal{Q})|\le2r^2\alpha n$.
%{Let $\mathcal{Q}$ be a such $K_r$-tiling that covers all vertices of~\ref{v1} and  }
% as desired.%独立集版本的Turan thm

%If $s=r-1$, that is, we already have the vertices $v_1,v_2,\cdots,v_{r-1}$, then we only need to choose a vertex in the common neighborhood of $v_1,v_2,\cdots,v_{r-1}$ in $B$. Note that $B$ has no $\gamma_{r}$-independent set and $d(v,B)\geq \beta n$ for each vertex $v\in V(G)$. As $v_j\in A_j$ is a $(\beta,A_j)$-bad vertex for each $j\neq i$ and $\beta$, we have \[
%|G'|=d_B(v_1,v_2,\cdots,v_{r-1})\ge\beta n-(r-1) \alpha^{1/5}n>\frac{\beta n}{2}>0,
%\]
%which implies that we can find $v_r\in B$ such that $\{v_1,v_2,\cdots,v_r\}$ induced a $K_r$ in $G$.

 %{For any $v\in A_i$ of~\ref{v1}, since there are at least $\frac{\beta}{2}n-2r\alpha n$ available vertices in each $A_{\ell}$ with $\ell\in [s]\setminus \{i\}$,}
%As $\alpha\ll\beta$, the above process yields a $K_r$-tiling $\mathcal{Q}$ that covers all vertices of~\ref{v1} and avoids vertices in $U$. Furthermore, each $K_r$ in $\mathcal{Q}$ has index vector $(1,\ldots,1,r-s)$.

%\medskip
%{\bf Covering bad vertices in $B$.}
%\medskip
Now, we consider vertices of~\ref{v2}. %Recall that in \ref{B06} there are at most $r\alpha n$ $(\alpha^{1/5},B)$-exceptional vertices in $B$.  %$(\alpha^{1/5},B)$-exceptional vertices. 
%The subsequent claim ensures that any vertex in $B$ lies in  a copy of  $K_r$ in $G[B\setminus (U\cup V(\mathcal{Q}))]$.

\begin{claim}\label{claim:good Kr-s in B}
   Let $W\subseteq V(G)$ be a set of size at most $\frac{\beta}{3}n$. For any $w\in B\setminus W$, there is a copy of $K_r$ in $G-W$ with index vector $(1,\ldots,1,r-s)$ containing $w$.  
\end{claim}

%\begin{claim}\label{claim:good Kr-s in B}
%For any $W\subseteq B$ with $|W|\le\beta n$, every vertex $w \in B \setminus W$ lies in a copy of $K_{r-s}$ that is disjoint from $W$.
%\end{claim}
\begin{proof}
Let $w\in B\setminus W$. We first show that the vertex $w$ lies in a copy of $K_{r-s}$ in $G[B\setminus W]$. 
The case for $r-s\le 1$ is trivial, so we only consider $r-s\geq 2$. Let $G_1$ be a graph obtained from $G[B\setminus W]$ by deleting all $(\alpha^{1/5},B)$-exceptional vertices in $G$. 
It suffices to find a copy of $K_{r-s-1}$ in $G_1[N_{G_1}(w)]$. Recall that in~\ref{B06} there are at most $r\alpha n$ $(\alpha^{1/5},B)$-exceptional vertices in $B$. Hence  
$$|N_{G_1}(w)|\ge \de(G[B])-\frac{\beta}{3} n-r\alpha n\ge \Big(1-\frac{1}{r-s}-2\beta\Big)\cdot{|B\setminus W|}.$$
For any vertex $u\in N_{G_1}(w)$, as $\beta\ll\frac{1}{r}$, we have 
\begin{align*}
|N_{G_1}(w)\cap N_{G_1}(u)|
&\ge |N_{G_1}(w)|+|N_{G_1}(u)|-|G_1|\\
&\ge |N_{G_1}(w)|+\Big(1-\frac{1}{r-s}-2\beta\Big)\cdot{|B\setminus W|}-|B\setminus W|\\
&= |N_{G_1}(w)|-\Big(\frac{1}{r-s}+2\beta\Big)\cdot{|B\setminus W|}\\
&\geq \Big(1-\frac{1}{r-s-2}-\beta\Big)\cdot|N_{G_1}(w)|,
\end{align*}
which implies $\de(G_1[N_{G_1}(w)])\ge \big(1-\frac{1}{r-s-2}-\beta\big)|N_{G_1}(w)|$. {Based on~\ref{B07} and $\beta\ll \gamma$, we obtain that $G_1[N_{G_1}(w)]$ contains no $\sqrt{\beta}$-independent set of size $\frac{|N_{G_1}(w)|}{r-s-2}$.  
Therefore, by applying Theorem~\ref{SuperT}  with $( G_1[N_{G_1}(w)], \beta)$ in place of $(G, \alpha)$,} there is a copy of $K_{r-s-1}$ in $G_1[N_{G_1}(w)]$, which together with $w$ forms a {copy of} $K_{r-s}$ in $G_1$. Let $T$ be {one such} $K_{r-s}$ with vertex set $\{w,u_1,\ldots,u_{r-s-1}\}$.

Observe that $d(u_i,A_j)\geq |A_j|-\alpha^{1/5}n$ for each $i\in [r-s-1]$ and each $j\in [s]$, and $d(w,A_j)\geq \beta n$ for each $j\in [s]$ by~\ref{B04}. By a similar construction with $\mathcal{Q}$ and the fact that $\alpha\ll \beta$, we can extend \( T\) into a copy of \( K_r \) with index vector $(1,\ldots,1,r-s)$ while avoiding $W$. %, denoted as $\mathcal{T}'$, such that for each $i\in [s]$, each \(K\in \mathcal{T}'\) consists of a single \( (\alpha^{1/5}, A_i) \)-good vertex and $V(K)\cap (U\cup V(\mathcal{Q}))=\emptyset$.
\end{proof}

{Recall that $|V(\mathcal{Q})|+|U|\leq 2r^2\alpha n+3r\alpha n$ and there are at most $r\alpha n$ $(\alpha^{1/5},B)$-exceptional vertices in $B$.  Hence, Claim~\ref{claim:good Kr-s in B} yields a $K_r$-tiling $\mathcal{T}$ in $G$ that avoids $V(\mathcal{Q})\cup U$ and covers all vertices of~\ref{v2}, where each copy of $K_r$ has index vector $(1,\ldots,1,r-s)$.} %Moreover, each clique in $\mathcal{T}$ contains exactly one $(\alpha^{1/5},B)$-exceptional vertex. 
Let $\mathcal{K}=\mathcal{T}\cup\mathcal{Q}$. Then $\mathcal{K}$ has index vector $|\mathcal{K}|(1,1,\ldots, r-s)$ and  $|V(\mathcal{K})\cup U|\leq 4r^2\alpha n$. %Let $\mathcal{P}':=\{A_1',\ldots,A_s',B'\}$ be a partition of $G':=G-V(\mathcal{K})$, which is obtained from $\mathcal{P}$ by deleting all vertices in $V(\mathcal{K})$.

\medskip
{\bf Step 2. Contracting disjoint copies of $K_{r-s}$ and $K_{r-s+1}$ in $G[B\setminus V(\mathcal{K})]$.}% so as to achieve~\ref{A1} in the resulting auxilary graph.}
\medskip

\noindent \textbf {Step 2.1. Fixing divisibility and balancing.
}

Note that it may happen that $(r-s) \nmid |B \setminus V(\mathcal{K})|$. Assume that $r-s\ge2$\footnote{The case $r-s=1$ is much simpler.} and $|B\setminus V(\mathcal{K})|\equiv q \pmod{(r-s)}$. %The final step is to deal with the case that the number of vertices in $B'$ is not divisible by $r-s$.  %\medskip
%{\bf Covering non-divisible vertices in $B$.}
%\medskip
%Fact~\ref{fact:bad vertices for Ai} and Fact~\ref{fact:bad vertices in B} imply that the covering process in the first two steps requires no more than $10r\alpha n$ vertices.
By~\ref{B05} and $\alpha\ll\beta$, we have
\begin{align}\label{mini-degree-K}
\de(G[B\setminus V(\mathcal{K})])\ge (r-s-1)n-r\alpha n-4r^2\alpha n\ge \Big(1-\frac{1}{r-s}-\beta\Big)\cdot|B\setminus V(\mathcal{K})|.
\end{align}
As $\eps\ll\beta$, applying Theorem~\ref{SuperT} with $(G[B \setminus V(\mathcal{K})],\beta)$ in place of $(G,\alpha)$ yields $\eps n$ vertex-disjoint copies of $K_{r-s+1}$, denoted as $\mathcal{H}_0$, in $G[B \setminus V(\mathcal{K})]$.  Let $H$ be a copy of $K_{r-s+1}$ in $\mathcal{H}_0$. %Denote $B''=B'\setminus V(H)$.  
Note that each vertex $v$ in $B\setminus V(\mathcal{K})$ satisfies $d(v,A_i)\ge |A_i|-\alpha^{1/5} n$ for all $i\in [s]$. 
{Thus, for any $j\in [s]$, we can extend $H$ into $K_r$ by iteratively select an $(\alpha^{1/5}, A_{i})$-good vertex $v_i\in A_i$ for all $i\in [s]\setminus\{j\}$, whilst avoiding the vertices in $V(\mathcal{K})\cup U$.}
%\sout{Thus, by avoiding vertices in $V(\mathcal{K})\cup U$, for any $j\in [s]$, we can iteratively select vertices \(v_1 \in A_1, \dotsc,v_{j-1}\in A_{j-1},v_{j+1}\in A_{j+1},\ldots, v_{s} \in A_{s}\) such that each \(v_i\) is in the common neighborhood of $V(K')$ and all previously chosen vertices.}
{This is possible because each chosen vertex \(v_i\) is \((\alpha^{1/5},A_i)\)-good, and the number of candidates for \(v_i\) is at least
\[
|A_i| - (r-s+1)\alpha^{1/5}n - s\alpha^{1/5}n-|V(\mathcal{K})\cup U| \ge \frac{n}{2}
.\]
}
%\sout{This can be achieved as every vertex $v_i$ is $(\alpha^{1/5}, A_{i})$-good for each $i\in [s]\setminus \{j\}$.}
%Then \(V(K') \cup \{v_i:i\in[s]\setminus \{j\}\}\) induces a \(K_r\) in \(G\). 
As $q < r-s$, we extend $q$ copies of $K_{r-s+1}$ from $\mathcal{H}_0$ as above into a $K_r$-tiling $\mathcal{H}$ and moreover%\footnote{We remark here that the choice of $\mathcal{H}$ would ensure that every large part from the final partition actually corresponds to a large part from the original $\mathcal{P}$.},
\begin{enumerate}
[label =\rm  (C\arabic{enumi})]
    \item\label{E1} if $|A_s| > n$, then $V(\mathcal{H}) \cap A_s = \emptyset$;
    \item\label{E2} if $|A_s|\leq n$, % and $p\in [s]$ is the largest integer such that $|A_{p}|\leq n$, 
   %\sout{ then $|V(\mathcal{H}) \cap A_i| = q$ for each $i\in [p+1,s]$, and}{(in this case $p=s$)} 
   then $|V(\mathcal{H}) \cap A_i| = q-q_i$ for $i\in [s]$ and some nonnegative integer $q_i$ such that $q_i\leq n-|A_i|$ and $\sum_{i\in [p]}q_i=q$.
    \end{enumerate}
%{Note that $|A_i|-(q-q_i)\le n-q$ with each $|A_i|<n$.}
Indeed, if $|A_s|\leq n$, then $p=s$ and  $\sum_{i\in [s]}(n-|A_i|)=|B|-(r-s)n\geq q$. Hence, for each $i\in [s]$, there exists $q_i\in \mathbb{N}$ such that $q_i\leq n-|A_i|$ and $\sum_{i\in [p]}q_i=q$. Thus, it suffices to choose for each $i\in [s]$,  $q_i$ copies of $K_r$ such that each of them contains no vertices from $A_i$.

Let $\hat{\mathcal{P}}:=\{\hat{A_1},\ldots,\hat{A_s},\hat{B}\}$ be the resulting partition of $\hat{G}:=G-V(\mathcal{K}\cup \mathcal{H})$, each part being obtained by removing all vertices in $V(\mathcal{K}\cup \mathcal{H})$. Then, $(r-s)\mid |\hat{B}|$. {Moreover,
\begin{itemize}
    \item if $|A_s|>n$, then $|\hat{A_i}|-\frac{|\hat{G}|}{r}=|A_i|-n$ for each $i\in [s-1]$ and $|\hat{A_s}|-\frac{|\hat{G}|}{r}=|A_s|-n+q$;
    \item if $|A_s|\leq n$, then $|\hat{A_i}|-\frac{|\hat{G}|}{r}=|A_i|-n+q_i\leq 0$ for each $i\in [s]$. % and $|\hat{A_i}|-\frac{|\hat{G}|}{r}=|A_i|-n$ for each $i\in [p+1,s]$. 
\end{itemize} 
It follows that $|\hat{A_i}|>\frac{|\hat{G}|}{r}$ {if and only if} $|A_i|>n$ for each $i\in [s]$.} %Therefore, $|\hat{A_i}|\leq \frac{|\hat{G}|}{r}$ for all $i\in [p]$.}

{Let $a:=\frac{|\hat{G}|}{r}-\frac{|\hat{B}|}{r-s}$. Suppose that  $a>0$. Note that $a\leq 2r\alpha n$ and  
$$
\sum_{i\in [p+1,s]}\Big(|\hat{A_i}|-\frac{|\hat{G}|}{r}\Big)\geq\sum_{i\in [s]}\Big(|\hat{A_i}|-\frac{|\hat{G}|}{r}\Big)=|\hat{G}|-|\hat{B}|-s\frac{|\hat{G}|}{r}=(r-s)\Big(\frac{|\hat{G}|}{r}-\frac{|\hat{B}|}{r-s}\Big)=(r-s)a.
$$
%Thus, 
%$$
%\sum_{i\in [p+1,s]}\big(|A_i|-n+q\big)\geq \sum_{i\in [p+1,s]}\Big(|\hat{A_i}|-\frac{|\hat{G}|}{r}\Big)\geq (r-s)a.
%$$
Recall that for each $i\in [p+1,s]$,  $G[\hat{A_i}\cap U]$ contains a matching of size $|A_i|-n+r\geq |\hat{A_i}|-\frac{|\hat{G}|}{r}+r-q$. For  each such $i$, choose $p_i\in \mathbb{N}$ such that $p_i\leq |\hat{A_i}|-\frac{|\hat{G}|}{r}$ and  $\sum_{i\in [p+1,s]}p_i=(r-s)a$. Let $U'\subseteq U$ be a subset for which each induced subgraph $G[\hat{A_i}\cap U']$ contains $p_i$ edges. %Notice that every vertex in $A_i\cap U'$ has at least $|D|-2\beta' n$ neighbors in every part $D\in \mathcal{P}\setminus \{A_i\}$.
Hence, by the same argument as in Claim \ref{covering-v}, one may find a $K_r$-tiling $\mathcal{R}$ consisting of exactly $(r-s)a$ copies of $K_r$ each containing one edge in $G[U']$, $r-s-1$ vertices from $B$ and at least one vertex from each $A_i$ with $i\in[s]$, whilst avoiding $V(\mathcal{K}\cup \mathcal{H})\cup (U\setminus U')$.  If $a\leq 0$, then take $\mathcal{R}=\emptyset$ and $p_i=0$ for each $i\in [p+1,s]$. Thus, $|V(\mathcal{K}\cup \mathcal{H}\cup \mathcal{R})|\leq 3r^3\alpha n$.}

{Let $\mathcal{P}':=\{A_1',\ldots,A_s',B'\}$ be the resulting partition of $G':=G-V(\mathcal{K}\cup \mathcal{H}\cup \mathcal{R})$, each part being obtained by removing all vertices in $V(\mathcal{K}\cup \mathcal{H}\cup \mathcal{R})$. Then, $(r-s)\mid |B'|$. 
Furthermore, if $a>0$, we have 
\begin{align}\notag
&\frac{|B'|}{r-s}=\frac{|\hat{B}|-(r-s-1)(r-s)a}{r-s}=\frac{|\hat{G}|}{r}-(r-s)a=\frac{|G'|}{r},
\\\notag
&|A_i'|=|\hat{A_i}|-(r-s)a\leq \frac{|\hat{G}|}{r}-(r-s)a=\frac{|G'|}{r}\ \text{for}\ i\in [p],\\\notag
&|A_i'|=|\hat{A_i}|-p_i-(r-s)a\ge \frac{|\hat{G}|}{r}-(r-s)a=\frac{|G'|}{r}\ \text{for}\ i\in [p+1,s].
\end{align}
Thus, in a summary, we have
\begin{align}\label{G'}
\frac{|B'|}{r-s}\ge \frac{|G'|}{r},\ |A_i'|\leq \frac{|G'|}{r}\ \text{for}\ i\in [p],\ \text{and}\ 
\frac{|G'|}{r}\leq |A_i'|\leq |\hat{A_i}|-\frac{|V(\mathcal{R})|}{r}\ \text{for}\ i\in [p+1,s]. 
\end{align}
Moreover, for each $i\in [p+1,s]$, $G'[A_i'\cap U]$ contains a matching of size} \begin{align}\label{matching'}
    |A_i|-n+r-p_i\geq |\hat{A_i}|-\frac{|\hat{G}|}{r}-p_i+r-q\geq  |A_i'|-\frac{|G'|}{r}+1.
\end{align}

Let $b:=\frac{|B'|}{r-s}- \frac{|G'|}{r}$. Clearly, $b\geq 0$. If $b>0$, then we shall take extra copies of $K_{r-s+1}$ from $\mathcal{H}_0$ which would be contracted into edges so as to achieve~\ref{A2}. It follows from~\ref{B1} that $b\leq r\alpha n$.  
%If $r-s=1$, then there are at least $\eps n\ge (r-s)b$ matching edges in $G'[B']$ as {$e(G[B])\ge \gamma_2n^2$ and $\eps \ll \gamma_2$}.
%If $r-s\ge 2$, t
Then as $\alpha \ll \eps$, we can greedily take $(r-s)b$ vertex-disjoint copies of $K_{r-s+1}$ from $\mathcal{H}_0-V(\mathcal{H}\cup \mathcal{R})$, denoted as $\mathcal{F}$. If $b= 0$, then take $\mathcal{F}=\emptyset$. Let $B'':=B'\setminus V(\mathcal{F})$. Clearly, $(r-s)\mid |B''|$.

%{Tell what you are doing next}
\medskip
\noindent {\bf Step 2.2. Construct a $K_{r-s}$-factor in $G'[B'']$.} 

The case $r-s=1$ is trivial. For $r-s\ge 2$, it follows from~\ref{B07} that $G'[B'']$ admits no ${\gamma}^2$-independent set of size $\frac{|B''|}{r-s}$. It is routine to check that 
\begin{align}\label{mini-degree}  \de(G'[B''])\ge \Big(1-\frac{1}{r-s}-\beta\Big)\cdot|B''|.
\end{align}
If $r-s\ge 3$, then Theorem~\ref{thm:Non-extremal case r>3} implies that there is a $K_{r-s}$-factor in $G'[B'']$ {since $\beta\ll\gamma$}, as desired.

Now, we consider \( r - s = 2 \). By Lemma~\ref{thm:r-s=2}, we conclude that either \(G'[B'']\) contains a perfect matching (in which case we are done), or it is disconnected with two odd components. %In fact, %we may assume the former occurs; otherwise 
%if the latter case holds, then a perfect matching can be obtained via a minor modification of the subset $B''$; as explained below.
Suppose we are in the latter case that $G'[B'']=G_1\cup G_2$, where both $|G_1|$ and $|G_2|$ are odd. It follows from~\eqref{mini-degree} that $|G_1|,|G_2|\geq \delta(G'[B''])\geq n-3\beta n$. Then each $G_i$ contains a triangle, say $T_i$ for $i\in [2]$.
Moreover,  by the argument in Claim~\ref{claim:good Kr-s in B}, each $T_i$ can be extended to a copy of \(K_r\), denoted by {\(K^{i,j}\)}, by using exactly one $(\alpha^{1/5},A_\ell)$-good vertex from each \({A_{\ell}'}\) with {\(\ell \in [s]\setminus\{j\}\)}, whilst avoiding vertices in ${U}$. Next we proceed according to the following two possible cases. 

\medskip
{\bf Case 1. $|A_s|\geq n$.} % or $|A_s|=n$ and $A_s\cap U$ contains no vertices of \ref{v1}.}
\medskip

Recall that $G'[A_s'\cap U]$ contains a matching of size {$|A_s'|-\frac{|G'|}{r}+1>0$} when $|A_s|>n$. Together with \ref{B2}, there is an edge $v_sv_s'$ in $G'[A_s'\cap U]$. {Notice that $A_s'\cap U$ may contain a vertex that is not 
$(\alpha^{1/5},A_s)$-good  if $|A_s|=n$.  Without loss of generality, if such a vertex exists, let it be $v_s$. %Notice that  $\Delta(G[A_s])\leq \beta'n$ if $|A_s|>n$.
By~\ref{B2} and~\ref{B04}, it holds that   $d(v_s,A_i)\geq \beta  n$ and $d(v_s',A_i)\geq |A_i|-2\beta' n$ for all $i\in [s-1]$.}  Following the argument in Claim~\ref{covering-v}, by avoiding vertices in $V(\{K^{1,s},K^{2,s}\})\cup U$, we can find an \((\alpha^{1/5},A_j)\)-good vertex   
\(v_j\in A_j'\) for each \(j\in[s-1]\) and a vertex \(w\in B''\) such that  
\(G'[\{v_1,\dots,v_{s-1},v_s,v_s',w\}]\) forms a clique \(K^1\).  Without loss of generality, assume that $w\in V(G_1)$. We then add $K^1\cup K^{2,s}$ back into $\mathcal{K}$. %This preserves the `balance' condition $2|V(K^1\cup K^{2,s})\cap A_i| = |V(K^1\cup K^{2,s})\cap B|$ for every $i\in [s]$.
By Lemma~\ref{thm:r-s=2}, the updated subgraph $G'[B''] = (G_1 - w) \cup (G_2 - V(K^{2,s}))$  contains a perfect matching, as desired.

\medskip
{\bf Case 2. $|A_s|<n$.}
\medskip

{Recall that $|A_1|\le \cdots\le |A_s|$. In this case,  $|B|>(r-s)n$. Observe that the $K_r$-tiling  $\mathcal{K}$ has index vector $|\mathcal{K}|(1,\ldots,1,r-s)$. Consequently, in Step $2.1$, either $q>0$  or $|B'|>(r-s)\frac{|G'|}{r}$ holds, which means \(\mathcal{H} \cup \mathcal{F} \ne \emptyset\).}

{Recall that $\mathcal{H}$ and $\mathcal{F}$ are families of vertex-disjoint copies of $K_r$ and $K_3$ obtained in Step~$2.1$, respectively.}
Choose a triangle $K\in \mathcal{F}$ (if $\mathcal{F}\neq \emptyset$, otherwise, choose from a {copy of} $K_r$ in $\mathcal{H}$) with vertex set $\{w_1,w_2,w_3\}\<B$. By~\eqref{mini-degree-K}, we know that $w_1$ is adjacent to some vertex in $V(G_1\cup G_2)$. Without loss of generality, assume that $N(w_1)\cap V(G_1)\neq \emptyset$.  If $K\in\mathcal{F}$, then update $\mathcal{F}$ by replacing the triangle $K$ with $T_2$ in $G_{2}$. By Lemma~\ref{thm:r-s=2}, the updated subgraph induced by  {$ (V(G_1\cup G_2)\setminus   V(T_2))\cup \{w_1,w_2,w_3\}$} has a perfect matching, as desired. If $K$ is chosen from some $K_r$ in $\mathcal{H}$, denoted as $K^2$, then update $\mathcal{H}$ by replacing $K^2$ with $K^{2,j}$ for some $j$ such that they have the same index vector {with respect to $\mathcal{P}'$}. By Lemma~\ref{thm:r-s=2}, we obtain that the current subgraph induced by  $(V(G_1\cup G_2)\setminus V(K^{2,j}))\cup \{w_1,w_2,w_3\}$ contains a perfect matching, as desired. \medskip

It is easy to see that the updated tiling  $\mathcal{K}$ has index vector $|\mathcal{K}|(1,\ldots,1,r-s)$, while the index vectors of updated $\mathcal{H}$ and $\mathcal{F}$ stay unchanged. Moreover,  $\mathcal{K}\cup \mathcal{H}\cup \mathcal{R}$ uses at most one edge in $G'[U\setminus U']$. In all possible cases, \(G'[B'']\) contains a $K_{r-s}$-factor, which is denoted as $\{K^1,\ldots,K^{t}\}$ for $t:=\frac{|B''|}{r-s}$. Let  $\mathcal{F}=\{K^{t+1},\ldots,K^{t
+(r-s)b}\}$ (could be empty if $b=0$).

\medskip
\noindent {\bf Step 2.3. Contracting and then verifying~\ref{A1}-\ref{A2}.}

 % there exist $t:=\frac{|B'|}{r-s}$ vertex-disjoint copies of $K_{r-s}$, say $K^1,\ldots,K^{t}$, in $G'[B']$, and $q$ vertex-disjoint copies of $K_{r-s+1}$, say $K^{t+1},\ldots,K^{t+q}$, in $H$. 
Contracting each $K^i$ ($i\in [t]$) into a vertex $w_i$ and each $K^j$ ($j\in [t+1,t+(r-s)b]$) into an edge $w_jw_j'$, yields a new vertex set $B^*$ of size $t+2(r-s)b$. 
We then construct an auxiliary graph $G^*$ with $V(G^*)=\big (\bigcup_{i\in[s]}A_i'\big)\cup B^*$, and $xy\in E(G^*)$ if and only if one of the following holds:
\begin{itemize}
    \item $x\in A_i'$, $y\in A_j'$, and $xy\in E(G)$ for distinct $i,j\in[s]$; 
    \item $x=w_i\in B^*$, $y\in A_j'$, and $y\in N_G(V(K^i))$ for $j\in [s]$ and $i\in[t]$;
    \item $x\in \{w_j,w_j'\}\subseteq B^*$, $y\in A_i'$, and $y\in N_G(V(K^j))$ for $i\in [s]$ and $j\in[{t+1,t+(r-s)b}]$;
    \item $\{x,y\}=\{w_j,w_j'\}$ for some $j\in[t+1,t+(r-s)b].$
\end{itemize}
Note that $\mathcal{P}^*=\{A_1',\ldots,A_s',B^*\}$ is a partition of $G^*$.  We first estimate the sizes of $|G^*|$ and $|B^*|$ to verify~\ref{A1}. 
\medskip

% It is routine to check that 
\textbf{(i)~Verify $|B^*|$.}\medskip

%If $b<0$, then $\mathcal{F}=\emptyset$ and \begin{align}\label{G^*}
%|B^*|=\frac{|B'|}{r-s}<\frac{|G'|}{r},~~
%|G^*|=\sum_{i\in [s]}|A_i'|+|B^*|=|G'|-|B'|+|B^*|=|G'|-\frac{r-s-1}{r-s}|B'|.
%\end{align}
%It follows from~\eqref{G^*} that
%$$
%\frac{|G^*|}{s+1}-\frac{|G'|}{r}=\Big(\frac{1}{s+1}-\frac{1}{r}\Big){|G'|}-\frac{r-s-1}{(r-s)(s+1)}|B'|=\frac{r-s-1}{s+1}\Big(\frac{|G'|}{r}-\frac{|B'|}{r-s}\Big)=-\frac{(r-s-1)b}{s+1}\geq 0,
%$$
%which implies that $\frac{|G^*|}{s+1}\geq \frac{|G'|}{r}$. Thus, $|B^*|\leq \frac{|G^*|}{s+1}$. \medskip

%For the case $b\ge 0$, we have
Notice that 
\begin{align*}%\label{B^*}
|B^*|=\frac{|B''|}{r-s}+2(r-s)b
&=\frac{|B'|-(r-s+1)(r-s)b}{r-s}+2(r-s)b \nonumber\\ 
&=\frac{|B'|}{r-s}+(r-s-1)b\nonumber \\ 
&=\frac{|G'|}{r}+(r-s)b.
\end{align*}
%Note that $|B^*|=\frac{|B''|}{r-s}+2(r-s)b$ and $|B''|=|B'|-(r-s+1)(r-s)b$. Together with $b=\frac{|B'|}{r-s}- \frac{|G'|}{r}$, we have
%\begin{align*}
%|B^*|=\frac{|B'|-(r-s)b\cdot(r-s+1)}{r-s}+2(r-s)b=\frac{|B'|}{r-s}+(r-s-1)b=\frac{|G'|}{r}+(r-s)b.
%\end{align*}
It follows that 
\begin{align}\label{b>0}
|G^*|&=\sum_{i=1}^s|A_i'|+|B^*|=\sum_{i=1}^s|A_i'|+\frac{|G'|}{r}+(r-s)\cdot\Big(\frac{|B'|}{r-s}- \frac{|G'|}{r}\Big) \nonumber\\
&=\sum_{i=1}^s|A_i'|+|B'|+\frac{1-r+s}{r}|G'|=|G'|+\frac{1-r+s}{r}|G'|=(s+1)\frac{|G'|}{r}.
\end{align}
 %If $b<0$, then there is no $F$. Together with $b=\frac{|B'|}{r-s}- \frac{|G'|}{r}<0$, we have $|B^*|=\frac{|B'|}{r-s}<\frac{|G'|}{r}$.
%Note that $|G^*|=\sum_{i=1}^s|A_i'|+\frac{|B'|}{r-s}$ and $|G'|=\sum_{i=1}^s|A_i'|+|B'|$.
%This implies $|G^*|=|G'|-\frac{r-s-1}{r-s}|B'|$.
%, where $n^*=|G^*|<\frac{2(s+1)n}{r}.$    Therefore, $|B^*|\leq \frac{|G^*|}{s+1}$. Furthermore,
%$$
%|B^*|-\frac{|G^*|}{s+1}=\frac{|B'|}{r-s}-\frac{|G^*|}{s+1}=\frac{|G'|-|G^*|}{r-s-1}-\frac{|G^*|}{s+1}=\frac{1}{r(r-s-1)}\Big(\frac{|G'|}{r}-\frac{|G^*|}{s+1}\Big)\geq-\frac{2r\alpha n'}{r-s-1}.
%$$
As $b\leq r\alpha n$ and $\frac{|G|}{r}\leq 2\frac{|G^*|}{s+1}$, \eqref{b>0} implies that 
\begin{align}\label{equal}
\frac{|G^*|}{s+1}= \frac{|G'|}{r}\ \text{and}\  |B^*|-\frac{|G^*|}{s+1}=(r-s)b\leq r^2\alpha n\leq 2r^2\alpha \frac{|G^*|}{s+1}.
\end{align}
Hence, $B^*$ satisfies~\ref{A1} {with parameter $2r^2\alpha$ in place of $\alpha$}.\medskip
%-2r^3\alpha n'\leq |B^*|-\frac{|G^*|}{s+1}=(r-s)b\leq 2r^3\alpha n'$. 

%By the construction of $\mathcal{K}\cup \mathcal{H}$, one has $\frac{|G'|}{r}=\frac{|G|-|V(\mathcal{K}\cup \mathcal{H})|}{r}=n-q-\frac{|V(\mathcal{K})|}{r}$. Moreover, the following hold:
\textbf{(ii)~Verify $|A_i'|$.}\medskip

%By the construction of $\mathcal{K}\cup \mathcal{H}$, one has $$\frac{|G^*|}{s+1}\geq \frac{|G'|}{r}=\frac{|G|-|V(\mathcal{K}\cup \mathcal{H})|}{r}=n-q-\frac{|V(\mathcal{K})|}{r}.$$ Therefore, 
By \eqref{G'} and \eqref{equal}, for each $i\in [p]$ we have 
\begin{align*}%\label{eq:A1}
|A_i'|-\frac{|G^*|}{s+1}\leq \frac{|G'|}{r}- \frac{|G^*|}{s+1}=0,%|A_i|-\frac{|V(\mathcal{K}\cup \mathcal{H})|}{r}-\frac{|G'|}{r}=,
\end{align*}
for each $i\in [p+1,s]$  we have 
\begin{align}\notag
|A_i'|-\frac{|G^*|}{s+1}&=|A_i'|-\frac{|G'|}{r}\leq |A_i|-\frac{|V(\mathcal{K})|}{r}-\frac{|V(\mathcal{R})|}{r}-\frac{|G|-|V(\mathcal{K}\cup \mathcal{H}\cup \mathcal{R})|}{r}\\\notag
&=|A_i|-n+q\leq 2\alpha n\leq 3\alpha  \frac{|G^*|}{s+1}. 
\end{align}

In a summary, $G^*$ satisfies $|A_1'|,\ldots, |A_{s}'|,|B^*|\leq \frac{|G^*|}{s+1}+2r^2\alpha \frac{|G^*|}{s+1}$ as desired in~\ref{A1} under the partition $\mathcal{P}^*$.  %{Tell what you are doing next}

%Observe that $|V(\mathcal{K}\cup\mathcal{H})\cap B|\geq (r-s)|V(\mathcal{K}\cup\mathcal{H})\cap A_i|$ for each $i\in [s]$, and $|V(\mathcal{K}\cup\mathcal{H})\cap A_i|\geq |V(\mathcal{K}\cup\mathcal{H})\cap A_j|$ for distinct $i,j\in [s]$ with $|A_i|\leq |A_j|$. 

\medskip

%\item if $|A_s|=n$, then for each $i\in [s]$ we have 
%$$
%|A_i'|-\frac{|G^*|}{s+1}=|A_i|-\frac{|V(\mathcal{K})|}{r}-(q-(n-|A_i|))-\frac{|G|-|V(\mathcal{K}\cup \mathcal{H})|}{r}=0.
%$$
%\end{itemize}

\medskip
\textbf{(iii)~Verify~\ref{A3} and~\ref{A2}.}\medskip

Notice that after covering all bad or exceptional vertices in Step 1, each $D'\in \mathcal{P}'$ satisfies $\De(G'[D'])\le \max\{\be' n, \alpha^{1/5}n\}$. 
% For each $D'\in \mathcal{P}'$ and each $v\notin D'$, as $\alpha\ll\beta'$, we have 
% \begin{align*}%\label{DofQ'}
% d_{G'}(v,D')\ge |D'|-\beta' n -3r^3\alpha n\ge |D'|-2{\beta'} n.
% \end{align*} 
{As $\alpha\ll\beta'$}, for each vertex $v\in A_i'$ with $i\in [s]$ and each $D^*\in \mathcal{P}^*\setminus \{A_i'\}$ we have 
$$
d_{G^*}(v,D^*)\geq |D^*|-\beta' n-|V(\mathcal{K}\cup \mathcal{H}\cup \mathcal{R})|\geq \big(1-2\beta'\big)|D^*|; 
$$
for each vertex $v\in B^*$ and each $i\in [s]$ we have 
$$
d_{G^*}(v,A_i')\geq |A_i'|-(r-s)\beta' n-|V(\mathcal{K}\cup \mathcal{H}\cup \mathcal{R})|\geq \big(1-2(r-s)\beta'\big)|A_i'|. 
$$
Therefore,~\ref{A3} holds with parameter $2(r-s)\beta'$ in place of $\beta$. 

%for each $D^*\in \mathcal{P}^*$ and each $v\notin D^*$, 
% we have
%\begin{align}\label{D^*-degree}
%d_{G^*}(v,D^*)\geq |D^*|-{2(r-s)\beta' n~(more~details)}\geq \big(1-3(r-s)\beta'\big)|D^*|. 
%\end{align}

It follows from the discussion in \textbf{(ii)} that $|A_i'|>\frac{|G^*|}{s+1}$ only if $|A_i|>n$. 
%If  $|A_s|\leq n$, then $|A_i'|>\frac{|G^*|}{s+1}$ only if $|A_i|>n$ or $|A_i|=n$ and $i\in [t]$. Let $p\in [p]$ be the largest integer such that $|A_i|<n$. Then $q=\sum_{i\in [p]}(n-|A_i|)$. 
%Hence if $A_i$ is a small part in $\mathcal{P}$, then we have $|A_i'|\leq \frac{|G^*|}{s+1}$, that is, $A_i'$ is a small part in $\mathcal{P}'$. 
%Moreover, since $|B|\leq (r-s)n$ if $r-s\geq 2$, one has $|B^*|\leq \frac{|G^*|}{s+1}$, that is, $B^*$ is a small part of $\mathcal{P^*}$ if $r-s\geq 2$.
%{Clearly, in constructing \(\mathcal{K} \cup \mathcal{K}'\), condition \ref{B2} allows us to avoid using vertices from at least \(|A_i| - n\) matching edges in each \(G[A_i]\) with \(|A_i| > n\), \(i \in [s]\).}
Therefore,  
\begin{itemize}    
\item\label{C2} for each $A_i'$ with $|A_i'|>\frac{|G^*|}{s+1}$, $G[A_i']$ contains a matching of size at least $|A_i'|-\frac{|G'|}{r}= |A_i'|-\frac{|G^*|}{s+1}$, {by combining \eqref{matching'}, \eqref{equal} and the fact that $\mathcal{K}\cup \mathcal{H}$ uses at most one edge in $E(G'[U\setminus U'])$};
    \item $|B^*|= \frac{|G^*|}{s+1}+(r-s)b$ and $G^*[B^*]$ contains a matching of size $(r-s)b$. % if $r-s=1$, then there are at least $r\alpha n$ matching edges inside $G'[B']$ as $|B'|\geq n-100r^2\alpha n$ and $G'[B']$ contains no $\gamma'$-independent $n$-set. 
\end{itemize}
Thus \ref{A2} holds. Notice that $|B^*|\geq  \frac{|G^*|}{s+1}$. By Lemma~\ref{lem:balance}, $G^*$ has a  $K_{s+1}$-factor in which each copy of $K_{s+1}$ satisfies either $V(K_{s+1})\cap B^*=\{w_i\}$ for some $i\in [t]$,  or $V(K_{s+1})\cap B^*=\{w_j,w_j'\}$ for some $j\in [t+1,t+(r-s)b]$. Thus, such a $K_{s+1}$-factor in $G^*$ corresponds to a $K_r$-factor in $G'$. 
%{which yields a $K_r$-factor in $G'$.(??seems we should require that all edges in $B^*$ are covered by disjoint copies of $k_{s+1}$)} 
Together with $\mathcal{K}\cup \mathcal{H}\cup \mathcal{R}$, we obtain a $K_r$-factor in $G$, as desired.   % and we complete the proof.
\end{proof}
%\medskip
%{\bf Remark.} As the proof demonstrates of Lemma~\ref{good-partition}, condition~\ref{B2} is overly strong. In fact, Lemma~\ref{lem:balance} implies that the matching of size $|A_i|-n$ in each large part $A_i$ is sufficient to find a $K_r$-factor. Furthermore, condition~\ref{B6} can be relaxed to only require the presence of at least one edge between any two sets of size $n$ in $B$, which is enough to guide us in finding a perfect matching in $B$ when $r-s=2$.

\subsection{Proof of Lemma~\ref{vertex-cover}}
In this subsection, we prove Lemma~\ref{vertex-cover}, which establishes the existence of a subgraph with a large minimum vertex cover.
\begin{proof}[\bf Proof of Lemma~\ref{vertex-cover}]
Let $G$ be an $((r-1)n+1)$-regular graph on $rn$ vertices, and let $\{A_1,A_2,\ldots,A_r\}$ be a balanced partition of $V(G)$. Suppose that $A_i^1$ is a minimum vertex cover of $G[A_i]$ for each $i\in [r]$. Without loss of generality, assume that $|A_1^1|\geq |A_2^1|\geq \cdots \geq |A_r^1|$. Clearly, there exists a vertex cover $C_i$ of $G[A_i]$ with size $|A_2^1|$ for each $i\in [2,r]$. For convenience, let $C_1:=A_1^1$. Denote $x_i:=|C_i|$ and $B_i:=A_i\setminus C_i$ for each $i\in [r]$. Hence, $G[B_i]=\emptyset$ for each $i\in [r]$. We proceed by considering the following two possible cases. 
% ~{(I suggest to update $C_i$ instead of $A_i^1$, and use $B_i$ for $A_i\setminus C_i$)}

\medskip
{\bf Case 1. $e\Big(B_1,\bigcup_{j\neq 2}C_j\Big)\leq e\Big(C_2,\bigcup_{j\neq 1}B_j\Big)$.}
\medskip

Since $G$ is  $((r-1)n+1)$-regular, one has 
$$
e(C_2,B_1)\leq |C_2|((r-1)n+1)-e\Big(C_2,\bigcup_{j\neq 1}B_j\Big).
$$
Note that as $x_2=x_3=\cdots=x_r$,
$$
e\Big(B_1,\bigcup_{j\neq 1}B_j\Big)\le\sum_{j\neq1}(n-x_1)(n-x_j)=(n-x_1)(r-1)(n-x_2).
$$
Therefore, we have
\begin{align*}
e\Big(B_1,\bigcup_{j\neq 2}C_j\Big)&=|B_1|((r-1)n+1)-e(B_1,C_2)-e\Big(B_1,\bigcup_{j\neq 1}B_j\Big)\\
&\geq (n-x_1)((r-1)n+1)-x_2((r-1)n+1)+e\Big(C_2,\bigcup_{j\neq 1}B_j\Big)-e\Big(B_1,\bigcup_{j\neq 1}B_j\Big)\\
&\geq (n-x_1-x_2)((r-1)n+1)+e\Big(C_2,\bigcup_{j\neq 1}B_j\Big)-(n-x_1)(r-1)(n-x_2).
\end{align*}
%where the last equality holds as $x_j=x_2$ for each $j\in [2,r]$.
By our assumption in this case, one has
\begin{align*}
(n-x_1-x_2)((r-1)n+1)\leq (n-x_1)(r-1)(n-x_2).
\end{align*}
It follows that $n\leq (r-1)x_1x_2+(x_1+x_2)\leq (r-1)x_1^2+2x_1.$ Thus, $x_1\geq \frac{1}{r}\sqrt{n}$, as desired.

\medskip
{\bf Case 2. $e\Big(B_1,\bigcup_{j\neq 2}C_j\Big)> e\Big(C_2,\bigcup_{j\neq 1}B_j\Big)$.}
\medskip

Notice that
$$
e\Big(\bigcup_{j\neq 2}C_j,\bigcup_{j\neq 1}B_j\Big)\leq \Big|\bigcup_{j\neq 2}C_j\Big|((r-1)n+1)-e\Big(\bigcup_{j\neq 2}C_j,B_1\Big).
$$
Therefore,\allowdisplaybreaks
\begin{align*}
e\Big(\bigcup_{j\neq 1}B_j,C_2\Big)=&\Big|\bigcup_{j\neq 1}B_j\Big|((r-1)n+1)-e\Big(\bigcup_{j\neq 1}B_j,\bigcup_{j\neq 2}C_j\Big)-e\Big(\bigcup_{j\neq 1}B_j,\bigcup_{t\in [r]}B_t\Big)\\
\geq &\Big|\bigcup_{j\neq 1}B_j\Big|((r-1)n+1)-\Big|\bigcup_{j\neq 2}C_j\Big|((r-1)n+1)+e\Big(\bigcup_{j\neq 2}C_j,B_1\Big)-\sum_{j\neq 1}\sum_{t\neq j}e(B_j,B_t)\\
\geq &((r-1)n+1)\Big((r-1)n-\sum_{j\neq 1}x_j-\sum_{j\neq 2}x_j\Big)+e\Big(\bigcup_{j\neq 2}C_j,B_1\Big)-\sum_{j\neq 1}\sum_{t\neq j}(n-x_j)(n-x_t)\\
=&((r-1)n+1)\Big((r-1)n-(r-1)x_2-x_1-(r-2)x_2\Big)+e\Big(\bigcup_{j\neq 2}C_j,B_1\Big)\\
&-(r-1)(n-x_2)\Big((r-1)n-x_1-(r-2)x_2\Big)\\
=&(r-1)n-x_1-(2r-3)x_2-(r-1)x_1x_2-(r-1)(r-2)x_2^2+e\Big(\bigcup_{j\neq 2}C_j,B_1\Big).
\end{align*}
By the assumption in this case, one has
$$
(r-1)n\leq x_1+(2r-3)x_2+(r-1)x_1x_2+(r-1)(r-2)x_2^2\leq (r-1)\big((r-1)x_1^2+2x_1\big).
$$
Thus, $x_1\geq \frac{1}{r}\sqrt{n}$, as desired. 
\end{proof}

%%%%%%%%%%%%%%%%%%%%%%%%%%%%%%%%%%%%%%%

\section{Concluding remarks}
%{Slightly discuss $G[p]$ for general $0<p<1$?}
%{A natural extension involves considering more general parameters for our problem. More precisely, let $G[p]$ be a random induced subgraph including each vertex independently with probability $p$ where $p\in (0,1)$ and we may consider a $d$-regular graph on $rn$ vertices and ask whether $G[p]$ admits a $K_r$-factor with a constant probability. On the one hand, we also consider an $((r-1)n+1)$-regular graph on $rn$ vertices and decrease $p$. Note that the probability that $G[p]$ admits a $K_r$-factor may be likely to asymptotically approach $0$ as $p\to 0$. Therefore, it is nature to ask for the point of $p$ where the behavior changes. On the other hand, we stick to $p=\tfrac{1}{2}$ and decrease $d$. In this case, we need to impose additional conditions to ensure that the random induced subgraph remains dense.}

{Dragani\'{c}, Keevash and M\"{u}yesser \cite{Keevash2025} proposed a plausible class of extremal construction for Conjecture~\ref{conj:K_r factor}: slightly unbalanced complete $r$-partite graphs with a suitable  factor in the largest part $A$.  They conjectured that the optimal constant $c$ should be  $\frac{1}{r^2}$, where in the random induced subgraph $G[S]$ we have one
probability factor of $\frac{1}{r}$ for $r$ dividing $|S|$ and another for $A\cap S$ being the largest part. In what follows, we present a construction to demonstrate that these two events alone are insufficient to guarantee the existence of a  $K_r$-factor in $G[S]$.}

{Assume that $r\geq 3$. We first consider an $rn$-vertex complete $r$-partite graph $G_0$ with vertex partition $A_1\cup \ldots \cup A_r$,  where $|A_1|=n+r-1$ and $|A_2|=\cdots=|A_r|=n-1$. Let $G$ be a graph obtained from $G_0$ by embedding an $r$-regular graph into  $A_1$ {such that $G[A_1]$ is triangle-free}. Clearly, $G$ is  $((r-1)n+1)$-regular and $G[A_i]$ is empty for every $i\in [2,r]$. Consider an arbitrary vertex subset $S\subseteq V(G)$. A necessary condition for $G[S]$ to contain a $K_r$-factor is that 
\begin{itemize}
    \item $r\mid |S|$;
    \item $\frac{|S|}{r}\le |A_1 \cap S|\le \frac{2|S|}{r}$;
   \item $|A_i \cap S|\le \frac{|S|}{r}$ for all $i \in [2,r]$. 
\end{itemize}
This implies that the optimal constant 
$c$ in Conjecture~\ref{conj:K_r factor} must be strictly less than $\frac{1}{r^2}$ {for every $r\ge 3$}. }

{In Theorem~\ref{thm:main thm} we assume that $G$ is $((r-1)n+1)$-regular, and we have remarked that the regularity requirement cannot be relaxed to a mere minimum degree condition. The need for regularity, however, arises only in Lemma~\ref{vertex-cover}. In fact, by carefully examining the proof of Lemma~\ref{vertex-cover}, we can replace the regularity assumption by a weaker degree condition: 
\begin{align}\label{mini-degree-G}
(r-1)n+1\leq \de(G)\leq \De(G)\leq (r-1)n+n^{0.5}.
\end{align}
Therefore, we can slightly extend Theorem~\ref{thm:main thm} under~\eqref{mini-degree-G}. Furthermore, one may also consider the more general setting of a $d$-regular graph where $d\ge(r-1)n+1$, and ask for determining the probability that its random induced subgraph $G[p]$ contains a $K_r$-factor, where $p\in(0,1)$.}

\vspace{0.5cm}
\textbf{Acknowledgement.}
The third author would like to thank Professor Jie Han for his valuable discussions and suggestions.
%{From here I wonder if we can prove Thm4.1 for almost regular graphs so as to extend our main result (Yang: after check the proof, here we can replace $n^{0.6}$ with $\alpha n$)} 

\bibliographystyle{abbrv}
\bibliography{ref.bib}
\end{document}